\documentclass{article}
\usepackage{amsmath}
\usepackage{amsfonts,amssymb,mathrsfs}
\usepackage{theorem}
\usepackage{hyperref}

\usepackage{tikz}
\usepgflibrary{arrows}
\usetikzlibrary{matrix}

\input amssym.def

\numberwithin{equation}{section}

\numberwithin{figure}{section}

\makeatletter
\newcommand\qedsymbol{\hbox{$\Box$}}
\newcommand\qed{\relax\ifmmode\Box\else
  {\unskip\nobreak\hfil\penalty50\hskip1em\null\nobreak\hfil\qedsymbol
  \parfillskip=\z@\finalhyphendemerits=0\endgraf}\fi}

\newenvironment{proof}[1][{}]{\par\noindent \underline{Proof of {#1}}. }{\qed}

\newenvironment{proofOld}[0]{\par\noindent \underline{Proof}.}{\qed}

\newcommand{\bfzero}{{\mathbf{0} }}

\newcommand{\Coll}{\mathsf{Coll}}

\newcommand{\Cobar}{{\mathrm{Cobar} }}

\newcommand{\Ger}{\mathsf{Ger}}
\newcommand{\BV}{\mathsf{BV}}

\newcommand{\calc}{\mathsf{calc}}
\newcommand{\Br}{\mathsf{Br}}
\newcommand{\KS}{\mathsf{KS}}
\newcommand{\CBr}{\mathsf{CBr}}

\newcommand{\Conv}{{\mathrm{Conv}}}

\newcommand{\Hom}{\mathrm{Hom}}

\newcommand{\Av}{\mathrm{Av}}

\newcommand{\Curv}{\mathrm{Curv}}

\newcommand{\Der}{\mathrm{Der}}
\newcommand{\id}{\mathrm{id}}

\newcommand{\ti}[1]{{\tilde{#1}}}
\newcommand{\wt}[1]{{\widetilde{#1}}}

\newcommand{\dia}{\diamond}

\newcommand{\ds}{\diamondsuit}

\newcommand{\al}{{\alpha}}

\newcommand{\bul}{{\bullet}}

\newcommand{\ms}{{\mathfrak{s}}}

\newcommand{\cF}{{\mathcal{F}}}
\newcommand{\pa}{{\partial}}

\newcommand{\bs}{{\mathbf{s}}\,}
\newcommand{\bsi}{{\mathbf{s}^{-1}\,}}
\newcommand{\bt}{{\mathbf{t}}}
\newcommand{\bq}{{\mathbf{q}}}

\newcommand{\cC}{\mathcal{C}}

\newcommand{\cL}{{\mathcal{L}}}
\newcommand{\cD}{\mathcal{D}}
\newcommand{\cA}{\mathcal{A}}
\newcommand{\cB}{\mathcal{B}}

\newcommand{\cG}{\mathcal{G}}
\newcommand{\cH}{{\mathcal{H}}}
\newcommand{\cZ}{{\mathcal{Z}}}

\newcommand{\cO}{\mathcal{O}}

\newcommand{\bbK}{{\mathbb K}}

\newcommand{\bbC}{{\mathbb C}}
\newcommand{\bbR}{{\mathbb R}}
\newcommand{\bbZ}{{\mathbb Z}}
\newcommand{\bbQ}{{\mathbb Q}}

\newcommand{\Op}{{\mathbb{OP}}}

\newcommand{\La}{{\Lambda}}

\newcommand{\D}{{\Delta}}

\date{}
\newtheorem{thm}{Theorem}[section]
\newtheorem{defi}[thm]{Definition}

\newtheorem{lem}[thm]{Lemma}
\newtheorem{cor}[thm]{Corollary}

\newtheorem{prop}[thm]{Proposition}
\newtheorem{claim}[thm]{Claim}
\newtheorem{cond}[thm]{Condition}
\newtheorem{example}[thm]{Example}

\theorembodyfont{\rm}
\newtheorem{remark}[thm]{Remark}

\title{When can a formality quasi-isomorphism over $\bbQ$ be constructed recursively?}

\author{V. A. Dolgushev and G.E. Schneider}

\date{}

\begin{document}

\large

\maketitle

\begin{abstract}
Let $\cO$ be a differential graded (possibly colored) operad defined over rationals. 
Let us assume that there exists a zig-zag of quasi-isomorphisms connecting 
$\cO \otimes \bbK$ to its cohomology, where $\bbK$ is any field extension of $\bbQ$.  
We show that for a large class of such dg operads, a formality quasi-isomorphism 
for $\cO$ exists and can be constructed recursively. Every step of our recursive procedure 
involves a solution of a finite dimensional linear system and it requires no explicit knowledge 
about the zig-zag of quasi-isomorphisms connecting $\cO \otimes \bbK$ to its cohomology.
\end{abstract}

\section{Introduction}
A differential graded (dg) operad $\cO$ is {\it formal} if there exists a sequence of 
quasi-isomorphisms (of dg operads)
$$
\cO \,\stackrel{\sim}{\leftarrow}\, \bullet  \,\stackrel{\sim}{\rightarrow}\, \bullet
\,\stackrel{\sim}{\leftarrow}\, \bullet ~ \dots ~ \bullet  \,\stackrel{\sim}{\rightarrow}\, H^{\bul}(\cO)
$$
connecting $\cO$ to its cohomology $H^{\bul}(\cO)$. 
Formality for dg operads (and other algebraic structures) is a subtle phenomenon. 
Currently, there are no effective tools for determining whether a given dg operad 
is formal or not. Moreover, in various interesting examples (including the braces operad 
$\Br$ \cite{Br}, \cite{K-Soi}, \cite{M-Smith}, its ``framed'' version $\CBr$ \cite{Campos}, \cite{Ward} 
and the Kontsevich-Soibelman operad $\KS$ \cite{K-Soi1}, \cite{Thomas-KS}) all known proofs of formality require 
transcendental tools \cite{K-mot}, \cite{LV-formality}, \cite{Dima-disc}, \cite{Thomas-KS}. 

In this paper we consider a dg operad $\cO$ defined over the field $\bbQ$ of rationals 
and assume that $\cO \otimes_{\bbQ} \bbK$ is formal for some field 
extension\footnote{In concrete examples, $\bbK = \bbR$ or $\bbC$.} $\bbK$ of $\bbQ$.
We consider a cobar resolution $\Cobar(\cC) \stackrel{\sim}{\to} H^{\bul}(\cO)$ of 
$H^{\bul}(\cO)$ and show that, under some mild conditions on $\cO$ and on the 
resolution $\Cobar(\cC)$, there is an explicit algorithm which allows us to produce 
a formality quasi-isomorphism\footnote{Recall that $\cO$ is formal if and only if there exists 
a quasi-isomorphism of dg operads \eqref{desired}.} 
\begin{equation}
\label{desired}
\Cobar(\cC)  ~\stackrel{\sim}{\longrightarrow} ~ \cO
\end{equation}
over $\bbQ$ recursively. The proof that this algorithm works is based on the existence of 
a sequence of quasi-isomorphisms connecting $\cO \otimes_{\bbQ} \bbK$ to its cohomology. 
However, no explicit knowledge about this sequence of quasi-isomorphisms is required 
at any step of this algorithm. 

We would like to mention that the existence of a formality quasi-isomorphism \eqref{desired} over 
$\bbQ$ (from the existence of a formality quasi-isomorphism over an extension of $\bbQ$)
was proved in paper \cite{Roig-plus} by F. Guill\'en Santos, V. Navarro, P. Pascual, 
and A. Roig. More precisely, see Theorem 6.2.1 in {\it loc. cit.}

Our paper is organized as follows. 
In Section \ref{sec:prelim}, we recall some basic concepts and 
fix the notational conventions.  
In Section \ref{sec:the-constr}, we introduce the concept of an MC-sprout, 
which can be viewed as an approximation to a formality quasi-isomorphism 
\eqref{desired}. Using this concept, we formulate the main theorem of this paper 
(see Theorem \ref{thm:main}) and deduce it from a technical lemma 
(see Lemma \ref{lem:betaalter}). Section \ref{sec:betaalter} is devoted to the proof 
of this lemma and Appendix \ref{app:lift} contains the proof of a useful lifting 
property for cobar resolutions. Finally, Appendix \ref{app:Tam-Arity4} displays
a third MC sprout in $\Conv(\Ger^{\vee}, \Br)$ which can be extended to a genuine 
MC element in $\Conv(\Ger^{\vee}, \Br)$. This MC sprout was found using the 
software \cite{Software} developed by the authors.

%

We should mention that our construction is inspired by Proposition 5.8 from classical 
paper\footnote{See also Theorem 4 and Corollary 4.1 in D. Bar-Natan's 
beautiful paper \cite{BN-GT}.} \cite{Drinfeld} by V. Drinfeld.

\vspace{0.28cm}

\noindent
\textbf{Acknowledgements:} The authors were partially supported by the NSF grant DMS-1501001. 
The authors are thankful to Sergey Plyasunov and Justin Y. Shi  for showing them how
to use the Python module {\it pickle.}  This module was used in the package \cite{Software} related 
to this paper.

\subsection{Preliminaries}
\label{sec:prelim}

In this paper, $\bbK$ is any field extension of the field $\bbQ$ of rational numbers 
and  $\otimes := \otimes_\bbQ$. For a cochain complex $V$, the notation $\cZ(V)$ is 
reserved for the subspace of cocycles. The degree of a vector $v$ in a graded vector 
space (or a cochain complex) $V$ is denoted by $|v|$. The notation $\bs$ (resp. $\bsi$) is 
reserved for the operator which shifts the degree up by $1$ (resp. down by $1$), i.e.
$$
(\bs V)^{\bul} = V^{\bul -1}\,, \qquad (\bsi V)^{\bul} = V^{\bul+1}\,.
$$
The notation $S_n$ is reserved for the symmetric group on $n$ letters. 

The abbreviation ``dg'' always means ``differential graded''. 

For a dg Lie algebra $\cL$, $\Curv$ is the map  $\Curv : \cL^1 \to \cL^2$
defined by the formula
\begin{equation}
\label{Curv-dfn}
\Curv(\al) : = \pa \al + \frac{1}{2} [\al, \al].  
\end{equation}
For example, Maurer-Cartan (MC) elements of $\cL$ are precisely 
elements of the zero locus of $\Curv$.  

Let us recall \cite{GMtheorem}, \cite{Getzler}, \cite{Hinich} that for every filtered dg Lie algebra $\cL$ 
(in the sense of \cite[Section 1]{GMtheorem}), the set of MC elements of  $\cL$ 
can be upgraded to a groupoid\footnote{This groupoid is 
actually a truncation of an $\infty$-groupoid (i.e. a fibrant simplicial set). However, for our purposes, we will not 
need cells of dimension $\ge 2$.} with MC elements being objects. Recall that 
two MC elements $\al, \ti{\al}$ of a filtered dg Lie algebra $\cL$ are isomorphic (in this groupoid) if 
there exists a degree $0$ element $\xi \in \cL$ such that 
\begin{equation}
\label{al-isom-ti-al}
\ti{\al} ~ = ~ \exp([\xi, ~]) \al  ~  - ~ \frac{\exp([\xi, ~]) - 1}{[\xi, ~]} \, \pa \xi,
\end{equation}
where the expressions $\exp([\xi, ~])$ and 
$$ 
\frac{\exp([\xi, ~]) - 1}{[\xi, ~]}
$$
are defined via the corresponding Taylor series\footnote{These series are well defined because 
$\cL = \cF_1 \cL$ and $\cL$ is complete with respect to the filtration.}.

In this paper, we will freely use the language of (colored) operads \cite{notes}, 
\cite{Fresse-book}, \cite{LV-book}. For a coaugmented cooperad $\cC$, the notation 
$\cC_{\circ}$ is reserved for the cokernel of the coaugmentation. 
For a dg pseudo-cooperad $P$, we denote by $P^{\ds}$ the dg cooperad 
which is obtained from $P$ by formally adjoining the counit. Clearly, for 
every coaugmented cooperad $\cC$, the cooperad $\cC_{\circ}^{\ds}$ is canonically 
identified with $\cC$. The notation $\Xi$ is reserved for the ordinal of colors. 
A ($\Xi$-colored)  {\it collection} $V$ is a family of cochain complexes 
$\{ V(\bq) \}_{\bq}$ indexed by all $\Xi$-colored corollas $\bq$ (with the standard labeling). 
For every $\Xi$-colored corolla $\bq$, $V(\bq)$ is equipped with the left action of 
the group 
$$
S_{k_1(\bq)} \times S_{k_2(\bq)} \times \dots \times S_{k_m(\bq)},  
$$ 
where $m$ is the total number of colors of the incoming edges and 
$k_i(\bq)$ is the number of incoming edges of the $i$-th color. 
For example, if the ordinal of colors $\Xi$ is the singleton, then
a collection is simply a family of cochain complexes $\{ V(n) \}_{n \ge 0}$, 
where each $V(n)$ is equipped with a left action of $S_n$.

The notation $\Coll$ is reserved for the category of $\Xi$-colored collections of 
graded vector spaces. For objects $Q_1, Q_2$ of $\Coll$ the notation 
$$
\Hom_{\Coll}(Q_1, Q_2)
$$ 
is reserved for the vector space of homomorphisms (of all degrees) from the collection 
$Q_1$ to the collection $Q_2$. For example, if the ordinal of colors is the singleton, then 
\begin{equation}
\label{Hom-Coll}
\Hom_{\Coll}(Q_1, Q_2) : =  \prod_{n \ge 0} \Hom_{S_n} \big(Q_1(n), Q_2(n)\big),
\end{equation} 
where
$$
\Hom_{S_n} \big(Q_1(n), Q_2(n)\big) = \Big( \Hom \big(Q_1(n), Q_2(n)\big) \Big)^{S_n}
$$
and $\Hom \big(Q_1(n), Q_2(n)\big)$ is the inner hom in the category of graded vector spaces. 

For a dg pseudo-cooperad $P$ and a dg operad $\cO$, the notation $\Conv(P, \cO)$
is reserved for the convolution Lie algebra \cite[Section 2.3]{stable}, \cite[Section 4]{notes}. 
The underlying graded vector space of $\Conv(P, \cO)$ is $\Hom_{\Coll}(P, \cO)$ and the 
Lie bracket is given by the formula
$$
[f,g] : = f \bullet g - (-1)^{|f| |g|} g \bullet f, 
$$
where $f \bullet g$ is the pre-Lie multiplication\footnote{See eq. (2.41) in \cite{stable}.} of $f$ and $g$
defined in terms of comultiplication on $P$ and multiplications on $\cO$. 

Let us recall \cite[Proposition 5.2]{notes} that 
MC elements of $\Conv(P, \cO)$ are in bijection with 
operad morphisms $F : \Cobar(P^{\ds}) \to \cO$. In particular, the operad morphism corresponding 
to a MC element $\al \in \Conv(P, \cO)$ will be denoted by $F_{\al}$. 

In this paper, we assume that 
\begin{cond}
\label{cond:P-filtered}
Every dg pseudo-cooperad $P$ carries an ascending filtration 
\begin{equation}
\label{P-circ-filtr}
\bfzero =  \cF^0 P \subset \cF^1 P \subset  \cF^2 P \subset  \cF^3 P \subset \dots 
\end{equation}
which is compatible with the differential and 
the comultiplications in the following sense: 
\begin{equation}
\label{diff-P-filtr}
\pa_{P}  \big( \cF^m  P \big) \subset  \cF^{m-1}  P, 
\end{equation}
\begin{equation}
\label{D-bt-filtr}
\D_{\bt} \big( \cF^m  P  \big) \subset 
\bigoplus_{m_1 + \dots + m_k  = m} 
\cF^{m_1} P  \otimes  \cF^{m_2} P  \otimes  \dots
\otimes  \cF^{m_k} P \,,
\end{equation}
where $\bt$ is a ($\Xi$-colored) planar tree with the set of leaves 
$\{1,2,\dots, n \}$ and $k$ nodal vertices.
Moreover, $P$ is cocomplete with respect to filtration \eqref{P-circ-filtr}, i.e.
\begin{equation}
\label{cocomplete}
P = \bigcup_{m} \cF^m P. 
\end{equation}
\end{cond}
\begin{remark}
\label{rem:Sul-alg}
Cobar resolutions $\Cobar(P^{\ds})$ for which $P$ satisfies 
Condition \ref{cond:P-filtered} may be thought of as 
analogs of Sullivan algebras from rational homotopy theory. 
Let us also mention that, due to \cite[Proposition 38]{MV11}, 
such dg operads $\Cobar(P^{\ds})$ are cofibrant. 
\end{remark}

For example, if the ordinal of colors $\Xi$ is the singleton, and 
$P(0) = P(1) = \bfzero$, then the filtration ``by arity'' 
\begin{equation}
\label{filtr-arity}
\cF^m P (n) : = \begin{cases}
P(n)  \qquad {\rm if} ~~ n \le m+1  \\
\bfzero  \qquad {\rm otherwise}.
\end{cases}
\end{equation}
satisfies Condition \ref{cond:P-filtered}. 

Condition \ref{cond:P-filtered} guarantees that, for every dg 
operad $\cO$,  the dg Lie algebra 
\begin{equation}
\label{Conv-P-cO}
\Conv(P, \cO)
\end{equation}
is equipped with the complete descending filtration: 
$$
\Conv(P, \cO)  = \cF_1 \Conv(P, \cO) \supset \cF_2 \Conv(P, \cO) \supset \dots
$$
\begin{equation}
\label{filtr-Conv}
\cF_m \Conv(P, \cO) : = \big\{ f\in  \Conv(P, \cO) ~\big|~ f \big|_{\cF^{m-1} P} = 0 \big\}.
\end{equation}
In other words, $\Conv(P, \cO)$ is a filtered dg Lie algebra in the sense of \cite[Section 1]{GMtheorem}.

\section{The recursive construction of formality quasi-isomorphisms}
\label{sec:the-constr}

Let $\cO$ be a dg operad and $\cH$ be the cohomology operad for $\cO$: 
$$
\cH : = H^{\bul}(\cO). 
$$
We assume that $\cH$ admits a cobar resolution $\Cobar(P^{\ds})$ where 
$P^{\ds}$ is a dg pseudo-cooperad satisfying Condition \ref{cond:P-filtered}.

Due to Corollary \ref{cor:zig-zag-shorter} from Appendix \ref{app:lift}, the 
problem of constructing a zig-zag of quasi-isomorphisms (of dg operads) connecting 
$\cO$ to $\cH$ is equivalent to the problem of constructing a single quasi-isomorphism 
(of dg operads)
$$
F : \Cobar(P^{\ds}) \to \cO. 
$$
The latter problem is, in turn, equivalent to the problem of constructing 
a MC element 
$$
\al \in \Conv(P, \cO)
$$
whose corresponding morphism $F_{\al} :  \Cobar(P^{\ds}) \to \cO$ is a 
quasi-isomorphism of dg collections.  

In this paper, we consider a dg operad $\cO$ and a cobar resolution
\begin{equation}
\label{rho-cH}
\rho \,: \, \Cobar(P^{\ds}) ~\stackrel{\sim}{\longrightarrow}~ \cH : = H^{\bul}(\cO). 
\end{equation}
We assume that the pair $(P, \cO)$ satisfies the following conditions: 
\begin{enumerate}
\item[\textbf{C1}] The dg pseudo-operad $P$ is equipped with an {\it additional} grading 
\begin{equation}
P = \bigoplus_{k \geq 1} \cG^k P\,, \qquad \cG^{\,\le 0} P = \bfzero
\end{equation}
which is compatible with the differential $\pa_P$ and the comultiplications $\D_{\bt}$ 
in the following sense: 
\begin{equation}
 \pa_P (\cG^k P) \, \subset \, \cG^{k-1} P,
\end{equation}
\begin{equation}
 \D_\bt(\cG^m P) ~\subset~ \bigoplus_{r_1+..+ r_q = m} \cG^{r_1}P \otimes \cG^{r_2}P \otimes ... \otimes \cG^{r_q}P,
\end{equation}
where $\bt$ is ($\Xi$-colored) tree with $q$ nodal vertices. 

\item[\textbf{C2}] $\cG^k P$ is finite dimensional for every $k$ and the graded 
components of $\cO(\bq)$ are finite dimensional for every $\Xi$-colored corolla $\bq$.  

\item[\textbf{C3}] The operad $\cH$ is generated by $\rho (\bs \cG^1 P)$ and 
\begin{equation}
\label{rho-for-k-ge2}
\rho \big|_{\bs  \cG^k P} ~ = ~0 \quad \forall ~~ k \ge 2. 
\end{equation}

\end{enumerate}
\begin{example} 
\label{ex:aritygrading}
Suppose that the ordinal of colors $\Xi$ is the singleton, 
$P(0) = P(1) = 0$ and the differential $\pa_P = 0$. Then the grading by 
arity 
\begin{equation}
\label{cG-arity}
\cG^k P (n) : = 
\begin{cases}
 P(n) \qquad {\rm if} ~~ n = k +1, \\
 \bfzero \qquad {\rm otherwise}.
\end{cases}
\end{equation}
satisfies Condition \textbf{C1}.
Moreover, if $P(n)$ is finite dimensional for all $n$ and 
each graded component of $\cO(n)$ is finite dimensional, then 
Condition \textbf{C2} is also satisfied.
In particular, for $P = \Ger^\vee_{\circ}$ the Koszul dual of the Gerstenhaber operad, 
and $\cO = \Br$ the braces operad \cite{Br}, \cite{DeligneTW}, all these assumptions are met.
\end{example}
\begin{remark}
\label{rem:CBr-KS}
Conditions \textbf{C1}, \textbf{C2}, and \textbf{C3} are also satisfied for the pairs
$(P_{\BV}, \CBr)$ and $(\calc^{\vee}, \KS)$, where $P_{\BV}$ is the dg pseudo-cooperad 
used for the cobar resolution \cite{BV} of the operad $\BV$ governing $BV$-algebras and 
$\calc^{\vee}$ shows up in the cobar resolution for the operad governing calculus 
algebras \cite{calc}, \cite[Definition 3]{HoCalc}.  
\end{remark}

\begin{remark}
\label{rem:cG-cF}
Clearly, every dg pseudo-operad $P$ with a grading $\cG^{\bul} P$ satisfying the above 
conditions has the ascending filtration 
\begin{equation}
\label{cF-P}
\cF^{m} P : = \bigoplus_{k \le m} \cG^k P
\end{equation}
and this filtration satisfies Condition \ref{cond:P-filtered}. 
\end{remark}
\begin{remark}
\label{rem:cG-k-P}
If we forget about the differential $\pa_P$ on $P$, every $\cG^k P$  
can be viewed as a collection of graded vector spaces. 
So we will tacitly identify 
elements in $\Hom_{\Coll}(\cG^k P , \cO)$
with elements $f \in \Conv(P, \cO)$ which satisfy the condition 
$f \big|_{\cG^m P} \equiv 0$ for all $m \neq k$. 
It is clear that $\Hom_{\Coll}(\cG^k P , \cO)$ is closed with respect to the differential 
$\pa_{\cO}$ on $\cO$ for every $k$. However, for the map $f \mapsto f \circ \pa_{P}$, we have 
$$
f \in  \Hom_{\Coll}(\cG^k P , \cO) ~ \mapsto~ f \circ \pa_P \in \Hom_{\Coll}(\cG^{k+1} P , \cO). 
$$
\end{remark} 
\begin{remark}
\label{rem:syzygy}
In many examples, the gradation on the (dg) pseudo-operad $P$ from Condition \textbf{C1} 
is precisely the syzygy gradation \cite[Appendix A]{BV}, \cite[Sections 3.3, 7.3]{LV-book}. 
\end{remark}

\hspace{0.5cm}

Let $F$ be an arbitrary morphism of dg operads
$$
F : \Cobar(P^{\ds}) \otimes \bbK \to \cO \otimes \bbK
$$
and $\pi_{\cH}$ be the canonical projection 
$$
\pi_{\cH} : \cZ(\cO) \to \cH 
$$ 
from the sub-operad $\cZ(\cO) : = \cO \cap \ker(\pa)$ to $\cH$. 

Since every vector in $\bs \cG^1 P$ is a cocycle in $\Cobar(P^{\ds})$ the restriction 
$F \big|_{\bs \cG^1 P}$ gives us a map of dg collections 
$$
F \big|_{\bs \cG^1 P} : \bs \cG^1 P \to \cZ(\cO). 
$$
We claim that
\begin{prop}
\label{prop:only-cG1-for-q-iso}
If the image of 
$$
\pi_{\cH} \circ F \big|_{\bs \cG^1 P}~ : ~ \bs \cG^1 P ~\to~ \cH
$$
generates the operad $\cH$ 
then $F$ is a quasi-isomorphism of dg operads.
The same statement holds if the base field $\bbQ$ is replaced 
by its extension $\bbK$. 
\end{prop}
\begin{proofOld}
Since all vectors in $\bs \cG^1 P$ are cocycles in $\Cobar(P^{\ds})$ and 
the sub-collection $\pi_{\cH} \circ F (\bs \cG^1 P)$ generates the operad 
$\cH$, the map 
$$
H^{\bul}(F) : H^{\bul}\big(\Cobar(P^{\ds})\big) \to \cH
$$
is surjective.

Since each graded component of $\cO(\bq)$ is finite dimensional for every corolla $\bq$
(see Condition \textbf{C2}), we know that each graded component of $\cH(\bq)$
is finite dimensional for every $\bq$. 

On the other hand, $H^{\bul}\big(\Cobar(P^{\ds})\big)$ is isomorphic to $\cH$. 

Thus the proposition follows from the fact a surjective map between isomorphic 
finite dimensional vector spaces is an isomorphism. 

Since this proof works for any base field (of characteristic zero), the last 
assertion in the proposition is obvious. 
\end{proofOld}

\subsection{MC-sprouts in $\Conv(P, \cO)$} 
\label{sec:MC-sprout}

\begin{defi}
\label{dfn:n-MC-sprout}
Let $\cF_{\bul}\Conv(P, \cO)$ be the descending filtration on $\Conv(P, \cO)$ 
coming from the ascending filtration \eqref{cF-P} on $P$ and $n$ be an integer $\ge 1$. 
An \emph{$n$-th MC-sprout} in $\Conv(P, \cO)$ is a degree $1$ element $\al \in \Conv(P, \cO)$ such that 
$$
\Curv(\al) \in \cF_{n+1} \Conv(P, \cO)
$$  
or equivalently 
\begin{equation}
\label{sprout-cG}
\Curv(\al) (X) = 0, \qquad \forall ~~ X \in \cG^{\,\le n} P.  
\end{equation}
\end{defi}
\begin{remark}
\label{rem:finite-sum}
Since $P$ is graded, every element $\al \in \Conv(P, \cO)$ can be uniquely written as 
$$
\al = \sum_{k=1}^{\infty} \al_k\,, \qquad \al_k \in \Hom_{\Coll}(\cG^k P, \cO).  
$$
Moreover, since $\al_{k}$ for $k > n$ do not contribute to the left hand side of 
\eqref{sprout-cG}, we may only consider $n$-th MC-sprouts of the form
$$
\al =  \sum_{k=1}^{n} \al_k\,, \qquad \al_k \in \Hom_{\Coll}(\cG^k P, \cO). 
$$
Due to our conditions on $\cO$ and $P$, any such MC-sprout is determined by 
a finite number of coefficients. 
\end{remark}
\begin{example}
\label{ex:truncation}
Let $\al$ be a genuine MC element of $\Conv(P, \cO)$. 
\emph{The $n$-th truncation} of $\al$ is the degree $1$ element  $\al^{[n]}$ 
of  $\Conv(P, \cO)$ defined by the formula
\begin{equation}
\label{trunc-dfn}
\al^{[n]} (X) = 
\begin{cases}
\al(X)  \qquad {\rm if} ~~ X \in \cG^{\,\le n} P \,, \\
 0  \qquad {\rm otherwise}\,.
\end{cases}
\end{equation}
Clearly, the $n$-th truncation of any MC element $\al$ of  $\Conv(P, \cO)$ is 
an $n$-th MC-sprout in $\Conv(P, \cO)$. It is also easy to see that the same 
formula \eqref{trunc-dfn} defines an $n$-th MC-sprout in $\Conv(P, \cO)$ provided 
$\al$ is an $m$-th MC-sprout and $m \ge n$.  We also call $\al^{[n]}$ \emph{the 
$n$-th truncation} of $\al$ even if $\al$ is not a genuine MC element of $\Conv(P, \cO)$
but merely an $m$-th MC-sprout for $m \ge n$. 
\end{example}
\begin{example}
\label{ex:Ger-Br-2nd-sprout}
Let $\Br$ be the braces operad and $T_{12}$,  $T_{\cup}$, $T_{1,23}$, 
$T^{\cup}_{123}$ and $T_{1, \bul, 23}$ be the brace trees shown in figures 
\ref{fig:Br-tree} and \ref{fig:Br-tree1}. Let $\al'$ be the following vector in 
$\Br(2) \otimes \La^{-2}\Ger(2)  \oplus \Br(3) \otimes \La^{-2}\Ger(3)$:
\begin{equation}
\label{al-pr}
\al' : =   T_{12} \otimes b_1 b_2  + \frac{1}{2}\, T_{\cup} \otimes \{b_1, b_2 \} + 
\frac{1}{2} \, T_{1,23} \otimes b_1 \{b_2, b_3\} - \frac{1}{3}\, T^{\cup}_{123} \otimes \{ b_1, \{ b_2, b_3\}\}
\end{equation}
$$
- \frac{1}{6} \,  T^{\cup}_{123} \otimes \{ b_2, \{ b_1, b_3\}\} 
- \frac{1}{6}\, T_{1, \bul, 23} \otimes \{ b_2, \{ b_1, b_3\}\} 
-\frac{1}{12} \, T_{1, \bul, 23} \otimes \{ b_1, \{ b_2, b_3\}\}. 
$$
A direct computation shows that $\Av(\al')$ is a second MC-sprout in $\Conv(\Ger^{\vee}, \Br)$. 
Here $\Av$ is the operator 
$$
\bigoplus_{n \ge 2} \Br(n) \otimes \La^{-2}\Ger(n) ~~\to~~ \bigoplus_{n \ge 2} 
\Hom_{S_n} \big(\Ger^{\vee}(n), \Br(n)\big)
$$
defined in eq. (C.3) in \cite[Appendix C]{DeligneTW} and, for $\al'$, we 
use the notation for vectors in  $\La^{-2}\Ger(n)$ from \cite[Section 4.3]{DeligneTW}.
\end{example}
\begin{figure}[htp] 
\begin{minipage}[t]{0.3\linewidth} 
\centering 
\begin{tikzpicture}[scale=0.6]
\tikzstyle{lab}=[circle, draw, minimum size=5, inner sep=1]
\tikzstyle{n}=[circle, draw, fill, minimum size=3]
\tikzstyle{root}=[circle, draw, fill, minimum size=0, inner sep=1]
\node[root] (rr) at (0, 0) {};
\node [lab] (v1) at (0,1) {$1$};
\node [lab] (v2) at (0,2.2) {$2$};
\draw (rr) edge (v1);
\draw (v1) edge (v2);
\end{tikzpicture}
\end{minipage}
~
\begin{minipage}[t]{0.3\linewidth} 
\centering 
\begin{tikzpicture}[scale=0.6]
\tikzstyle{lab}=[circle, draw, minimum size=5, inner sep=1]
\tikzstyle{n}=[circle, draw, fill, minimum size=3]
\tikzstyle{root}=[circle, draw, fill, minimum size=0, inner sep=1]
\node[root] (rr) at (0, 0) {};
\node [n] (n) at (0,1) {};
\node [lab] (v1) at (-0.8,1.8) {$1$};
\node [lab] (v2) at (0.8,1.8) {$2$};
\draw (n) edge (rr) edge (v1) edge (v2);
\end{tikzpicture}
\end{minipage}
~
\begin{minipage}[t]{0.3\linewidth} 
\centering 
\begin{tikzpicture}[scale=0.6]
\tikzstyle{lab}=[circle, draw, minimum size=5, inner sep=1]
\tikzstyle{n}=[circle, draw, fill, minimum size=3]
\tikzstyle{root}=[circle, draw, fill, minimum size=0, inner sep=1]
\node[root] (rr) at (0, 0) {};
\node [lab] (v1) at (0,1) {$1$};
\node [lab] (v2) at (-0.8,1.8) {$2$};
\node [lab] (v3) at (0.8,1.8) {$3$};
\draw (v1) edge (rr) edge (v2) edge (v3);
\end{tikzpicture}
\end{minipage}
\caption{The brace trees $T_{12}, T_{\cup}$, and $T_{1,23}$, respectively} \label{fig:Br-tree}
\vspace{0.5cm}
\begin{minipage}[t]{0.45\linewidth} 
\centering 
\begin{tikzpicture}[scale=0.6]
\tikzstyle{lab}=[circle, draw, minimum size=5, inner sep=1]
\tikzstyle{n}=[circle, draw, fill, minimum size=3]
\tikzstyle{root}=[circle, draw, fill, minimum size=0, inner sep=1]
\node[root] (rr) at (0, 0) {};
\node [n] (n) at (0,0.8) {};
\node [lab] (v1) at (-0.8,1.8) {$1$};
\node [lab] (v2) at (0,1.8) {$2$};
\node [lab] (v3) at (0.8,1.8) {$3$};
\draw (n) edge (rr) edge (v1) edge (v2) edge (v3);
\end{tikzpicture}
\end{minipage}
~
\begin{minipage}[t]{0.45\linewidth} 
\centering 
\begin{tikzpicture}[scale=0.6]
\tikzstyle{lab}=[circle, draw, minimum size=5, inner sep=1]
\tikzstyle{n}=[circle, draw, fill, minimum size=3]
\tikzstyle{root}=[circle, draw, fill, minimum size=0, inner sep=1]
\node[root] (rr) at (0, -0.3) {};
\node [lab] (v1) at (0,0.5) {$1$};
\node [n] (n) at (0,1.5) {};
\node [lab] (v2) at (-0.8,2.5) {$2$};
\node [lab] (v3) at (0.8,2.5) {$3$};
\draw (v1) edge (rr) (n) edge (v1) edge (v2) edge (v3);
\end{tikzpicture}
\end{minipage}
\caption{The brace trees $T^{\cup}_{123}$, and $T_{1, \bul, 23}$, respectively} \label{fig:Br-tree1}
\end{figure} 

Since all vectors in $\bs \cG^1 P$ are cocycles in $\Cobar(P^{\ds})$, 
$$
\al (\cG^1 P) \subset \cZ(\cO)
$$ 
for every MC-sprout $\al \in \Conv(P,\cO)$.
Let us observe that 
\begin{prop}
\label{prop:2nd-exists}
If $H^{\bul}(\cO) \cong \cH$ and $\al_{\cH}$ is the MC element 
in $\Conv(P, \cH)$ corresponding to \eqref{rho-cH},
then there exists a second MC-sprout $\al \in \Conv(P,\cO)$ such that the diagram 
\begin{equation}
\label{diag-al-rho}
\begin{tikzpicture}
\matrix (m) [matrix of math nodes, row sep=2.5em, column sep=2.5em]
{~~~ & \cZ(\cO)  ~ \\
 \cG^1 P &  \cH\\ };
\path[->, font=\scriptsize]
(m-1-2) edge node[right] {$\pi_{\cH}$} (m-2-2) 
(m-2-1) edge  node[auto] {$\al_{\cH}$} (m-2-2)
(m-2-1) edge  node[auto] {$\al $} (m-1-2);
\end{tikzpicture}
\end{equation}
commutes. 
\end{prop} 
\begin{proofOld}
Since we work with vector spaces, there exist splittings
\begin{equation}
\label{split}
\eta_{\bq} : \cH(\bq) \to \cZ(\cO(\bq))
\end{equation}
of the projections $\pi_{\cH} :  \cZ(\cO(\bq)) \to \cH(\bq)$ for every $\Xi$-colored corolla $\bq$. 

Moreover, since our base field has characteristic zero, we can use the 
standard averaging operators (for products of symmetric groups) and turn 
the splittings \eqref{split} into a map of collections 
\begin{equation}
\label{ms}
\ms : \cH \to \cZ(\cO)
\end{equation}
for which 
\begin{equation}
\label{pi-ms}
\pi \circ \ms  = \id_{\cH}\,. 
\end{equation} 

The similar argument, implies that there exists a map of collections
\begin{equation}
\label{ti-ms}
\ti{\ms} : \ker \big( \cZ(\cO) \to \cH \big) \to \cO
\end{equation}
which splits $\pa_{\cO}: \cO \to \ker \big( \cZ(\cO) \to \cH \big)$. 

By setting\footnote{Note that, due to \eqref{rho-for-k-ge2}, $\al^{(1)}(X) = 0$ for every $X \in \cG^{\ge 2} P$.} 
\begin{equation}
\label{al-first}
\al^{(1)} : =  \ms \circ \al_{\cH}
\end{equation}
we get a first MC-sprout in $\Conv(P, \cO)$ for which 
\begin{equation}
\label{pi-al-first}
\pi_{\cH} \circ \al^{(1)} = \al_{\cH}\,. 
\end{equation}

Let us observe that, since  $\al^{(1)}$ lands in $\cZ(\cO)$, 
the assignment 
$$
X \in \cG^2 P ~\mapsto~ \al^{(1)} \pa_{P} (X) + \al^{(1)} \bullet \al^{(1)} (X)
$$
gives us a map of collections:
\begin{equation}
\label{cG2-P-cO}
\cG^2 P \to \cZ(\cO).
\end{equation}

Since $\pi_{\cH}$ is compatible with the operadic multiplications, 
the composition of \eqref{cG2-P-cO} with $\pi_{\cH}$ sends 
$X \in \cG^2 P$ to  
\begin{equation}
\label{X-to-zero}
\al_{\cH}(\pa_{P} X) + \al_{\cH} \bullet \al_{\cH} (X) ~\in ~ \cH
\end{equation}

On the other hand, the vector \eqref{X-to-zero} is zero since $\al_{\cH}$ satisfies the 
MC equation and $\cH$ has the zero differential. 

Since the composition of \eqref{cG2-P-cO} with $\pi_{\cH}$ is zero, the map 
\eqref{cG2-P-cO} lands in $\ker \big( \cZ(\cO) \to \cH \big)$ and hence \eqref{cG2-P-cO} can 
be composed with the splitting \eqref{ti-ms}. 

Setting
\begin{equation}
\label{al-second}
\al (X) =
\begin{cases}
\quad \al^{(1)}(X)  \qquad {\rm if} ~~ X \in \cG^1 P \\[0.3cm]
- \ti{\ms} \big( \al^{(1)} \pa_{P} (X) + \al^{(1)} \bullet \al^{(1)} (X) \big)  \qquad {\rm if} ~~ X \in \cG^2 P \\[0.3cm]
 \quad 0 \qquad {\rm otherwise}
\end{cases}  
\end{equation}
we get a degree $1$ element $\al$ which satisfies
$$
\pa_{\cO} \al(X) + \al (\pa_P X) + \al \bullet \al (X) = 0  \qquad \forall ~~ X \in \cG^1 P \oplus \cG^2 P. 
$$
In other words, $\al$ is a second MC-sprout in $\Conv(P, \cO)$. 

Equations \eqref{pi-al-first} and \eqref{al-second} imply that the diagram \eqref{diag-al-rho} commutes.
\end{proofOld}
\begin{remark} 
\label{rem:then-q-iso}
Let $n$ be an integer $\ge 2$ and 
$$
\al = \al_1 + \al_2 + \dots + \al_n, \qquad \al_k \in \Hom_{\Coll}(\cG^k P, \cO)  
$$
be an $n$-th MC-sprout in $\Conv(P, \cO)$.
Proposition \ref{prop:only-cG1-for-q-iso} implies that, if $\al$ is a truncation of 
a genuine MC element $\beta \in \Conv(P, \cO)$ and  the diagram 
\eqref{diag-al-rho} commutes then the corresponding map of dg operads 
$$
F_{\beta} : \Cobar(P^{\ds}) \to \cO
$$ 
is a quasi-isomorphism. Thus, for our purposes, it makes sense to 
consider only MC-sprouts in $\Conv(P, \cO)$ for 
which the diagram \eqref{diag-al-rho} commutes. 
\end{remark}
\begin{remark}
\label{rem:non-formal}
Due to Proposition \ref{prop:2nd-exists}, a second MC-sprout $\al$ exists even if $\cO$ is non-formal. 
Of course, if $\cO$ is non-formal, such $\al$ is not a truncation of 
any MC element in $\Conv(P, \cO)$. 
\end{remark}

\subsection{The main theorem}
\label{sec:main-thm}
Let, as above, $\cO$ be a dg operad defined over $\bbQ$, $\cH : = H^{\bul}(\cO)$, and 
$$
\rho \,: \, \Cobar(P^{\ds}) ~\stackrel{\sim}{\longrightarrow}~ \cH
$$
be a cobar resolution for $\cH$, where $P$ is a dg pseudo-cooperad. 

The main result of this paper is the following theorem. 
\begin{thm} 
\label{thm:main}
Let us assume that the pair $(\cO, P)$ satisfies 
Conditions {\bf C1},  {\bf C2},  {\bf C3},
and $\cO \otimes \bbK$ is formal for some field extension $\bbK$ of $\bbQ$. 
Let, furthermore, $n$ be an integer $\ge 2$ and 
\begin{equation}
\label{al-n-th}
\al = \alpha_1 + \dots + \alpha_n \in \Conv(P, \cO), \qquad \al_k \in \Hom_{\Coll}(\cG^k P, \cO)
\end{equation}
be an $n$-th MC-sprout in $\Conv(P, \cO)$ for which 
the diagram \eqref{diag-al-rho} commutes. 
Then there exists an $(n+1)$-th MC-sprout $\ti{\al}$ such that 
$$
\ti{\al}_k  = \al_k \qquad \forall~~ k  < n. 
$$ 
Moreover, the unknown vectors $\ti{\al}_{n}$ and  $\ti{\al}_{n+1}$ can be found 
by solving a finite dimensional linear system.
\end{thm}
Theorem \ref{thm:main} has the following immediate corollaries: 
\begin{cor}
\label{cor:main}
Under the above conditions on the pair $(\cO, P)$, a quasi-isomorphism of operads 
\begin{equation}
\label{q-iso}
\Cobar(P^{\ds}) \, \stackrel{\sim}{\longrightarrow} \, \cO
\end{equation} 
can be constructed recursively. Moreover the algorithm for 
constructing \eqref{q-iso} requires no explicit knowledge 
about a sequence of quasi-isomorphisms (of operads) connecting 
$\cO \otimes \bbK$ to $\cH \otimes \bbK$.  \qed
\end{cor}
\begin{cor}
\label{cor:every}
If the assumptions of Theorem \ref{thm:main} hold and 
$$
\al = \alpha_1 + \dots + \alpha_n \in \Conv(P, \cO), \qquad \al_k \in \Hom_{\Coll}(\cG^k P, \cO)
$$
is an $n$-th MC-sprout in $\Conv(P, \cO)$ for which 
the diagram \eqref{diag-al-rho} commutes, then there exists 
a genuine MC element $\al_{MC} \in  \Conv(P, \cO)$ whose 
$(n-1)$-th truncation $\al_{MC}^{[n-1]}$ coincides with
$$
\al_1 + \dots + \al_{n-1}\,.
$$
\end{cor}

The proof of Theorem \ref{thm:main} is based on the following technical statement: 
%
\begin{lem} 
\label{lem:betaalter} 
Let $n$ be an integer $\ge 2$ and 
$$
\al = \al_1 + \al_2 + \dots + \al_n\,, \qquad \al_k \in \Hom_{\Coll}(\cG^kP, \cO)  
$$
be an $n$-th MC-sprout in $\Conv(P, \cO)$ for which the diagram 
\eqref{diag-al-rho} commutes. Then there exists a genuine MC element
$\beta \in \Conv(P, \cO \otimes \bbK)$ such that 
$$
\al_1 + \al_2 + \dots + \al_{n-1} = \beta^{[n-1]}\,, 
$$
where $\beta^{[n-1]}$ is the $(n-1)$-th truncation of $\beta$.
\end{lem}

\subsection{Theorem \ref{thm:main} follows from Lemma \ref{lem:betaalter}}
\label{sec:thm-proof}

Lemma \ref{lem:betaalter} is proved in Section \ref{sec:betaalter} below.
Here we show that Theorem \ref{thm:main} is a consequence of 
Lemma \ref{lem:betaalter}. 
  
Our goal is to find 
$$
\tilde{\alpha} := \tilde{\alpha}_1 + \tilde{\alpha}_2 + \dots + \tilde{\alpha}_{n+1}\,, \qquad \ti{\al}_k \in \Hom_{\Coll}(\cG^k P, \cO)
$$
satisfying 
$$
\ti{\al}_k = \al_k\,, \qquad \forall~~ k \le n - 1
$$
and 
\begin{equation} 
\label{Curv-zero}
\Curv(\tilde{\alpha})(X) = 0 \qquad \forall ~~ X \in \cG^{\,\le  n+1} P.
\end{equation}

So we set 
\begin{equation}
\label{al-k-le-1n}
\ti{\al}_k : = \al_k\,, \qquad \forall~~ k \le n - 1 
\end{equation}
and observe that the unknown terms $\ti{\al}_{n}$ and $\ti{\al}_{n+1}$ show up only in the equations 
\begin{equation} 
\label{MC0}
\pa_{\cO}\tilde{\alpha}_n(X) + \alpha_{n-1}(\pa_{P}X) 
+ \frac{1}{2}\sum_{\substack{i+j=n, \\[0.1cm] i,j \ge 1}}[\alpha_i,\alpha_j](X) = 0
\quad X \in \cG^{n} P,
\end{equation}
\begin{equation} 
\label{MC1A}
\pa_{\cO}\tilde{\alpha}_{n+1}(Y) + 
\tilde{\alpha}_n(\pa_{P} Y) + [\al_1, \ti{\al}_n](Y) + 
\frac{1}{2}\sum_{\substack{i+j=n+1 \\ i,j < n}}[\al_i,\al_j](Y) = 0
\quad Y \in \cG^{n+1} P.
\end{equation}
Moreover, the unknown terms enter these equations linearly. 

Due to the finite dimensionality condition (see {\bf C2}), equations 
\eqref{MC0} and \eqref{MC1A} can be viewed as a finite dimensional inhomogeneous 
linear system for the unknown vectors $\ti{\al}_{n}$ and $\ti{\al}_{n+1}$. 

Thanks to Lemma \ref{lem:betaalter}, there exists a genuine MC element in $\Conv(P, \cO \otimes \bbK)$
$$
\beta = \sum_{k=1}^{\infty} \beta_k  \qquad \beta_k \in \Hom_{\Coll}(\cG^kP, \cO \otimes \bbK)
$$
such that 
$$
\beta_k = \al_k, \qquad \forall~~ k \le n-1.  
$$

Therefore, the linear system corresponding to equations \eqref{MC0} and \eqref{MC1A}
has a solution over the field $\bbK$. Thus, since both the coefficient matrix and 
the right hand side of this linear system are defined over $\bbQ$, we have a solution
over $\bbQ$. 

Finally, equation $\Curv(\tilde{\alpha})(X) = 0$ is satisfied for every $X \in \cG^{\, \le n-1} P$, 
since $\ti{\al}_k : = \al_k $ for $k \le n-1$ and the original $\al$ is an $n$-th MC-sprout. \qed

\section{The proof of Lemma \ref{lem:betaalter}}
\label{sec:betaalter}

Let us first prove the following statement.
\begin{prop}
\label{prop:beta-needed}
Let $n$ be an integer $\ge 2$ and $\al$ be an $n$-th MC sprout
in $\Conv(P, \cO)$ for which the diagram \eqref{diag-al-rho} commutes.
Then there exists a MC element
$$
\beta = \sum_{k=1}^{\infty} \beta_k, \qquad \beta_k \in \Hom_{\Coll}(\cG^k P, \cO \otimes \bbK )
$$ 
in $\Conv(P, \cO \otimes \bbK)$ such that 
\begin{equation}
\label{cG-1-all-good}
\beta \big|_{\cG^1 P} ~ = ~ \al \big|_{\cG^1 P}.
\end{equation}
\end{prop}
\begin{remark}
\label{rem:F-beta-q-iso}
Proposition \ref{prop:only-cG1-for-q-iso} and Condition {\bf C3} imply that the operad morphism 
$$
F_{\beta} :  \Cobar(P^{\ds}) \otimes \bbK \to \cO \otimes \bbK
$$
corresponding to the above MC element $\beta$ is a quasi-isomorphism. 
\end{remark}
\begin{proof}[Proposition \ref{prop:beta-needed}]
Since $\cO \otimes \bbK$ is formal, there exists a quasi-isomorphism of dg operads
\begin{equation}
\label{F}
F : \Cobar(P^{\ds}) \otimes \bbK  \to \cO \otimes \bbK 
\end{equation}

Both $F$ and $\rho$ \eqref{rho-cH} induce the isomorphisms of operads
\begin{equation}
\label{H-F}
H^{\bul}(F) : H^{\bul} \big( \Cobar(P^{\ds}) \otimes \bbK \big) \to \cH \otimes \bbK 
\end{equation}
and 
\begin{equation}
\label{H-rho}
H^{\bul}(\rho) : H^{\bul} \big( \Cobar(P^{\ds}) \otimes \bbK \big) \to \cH \otimes \bbK.  
\end{equation}

Hence there exists (a unique) operad automorphism 
$$
T  :  \cH \otimes \bbK \to \cH \otimes \bbK
$$
such that 
\begin{equation}
\label{H-F-rho-T}
T \circ H^{\bul}(F) =  H^{\bul}(\rho).
\end{equation}

Due to Corollary \ref{cor:lift} from Appendix \ref{app:lift}, there exists a map of operads 
$$
\ti{T} ~:~ \Cobar(P^{\ds}) \otimes \bbK ~\to~ \Cobar(P^{\ds}) \otimes \bbK
$$
such that the diagram  
\begin{equation}
\label{diag-ti-T}
\begin{tikzpicture}
\matrix (m) [matrix of math nodes, row sep=2em, column sep=2em]
{  \Cobar(P^{\ds}) \otimes \bbK &  \Cobar(P^{\ds}) \otimes \bbK ~ \\
  \cH \otimes \bbK   &   \cH \otimes \bbK \\ };
\path[->, font=\scriptsize]
(m-1-1) edge node[above] {$\ti{T}$} (m-1-2)
edge node[left] {$\rho$} (m-2-1)
(m-1-2) edge  node[right] {$\rho$} (m-2-2)
(m-2-1) edge node[above] {$T$} (m-2-2) ;
\end{tikzpicture}
\end{equation}
commutes up to homotopy. 

Since $\rho \circ \ti{T}$ is homotopic to $T\circ \rho$, $\rho$ is a quasi-isomorphism, and 
$T$ is an automorphism of operads, $\ti{T}$ is a quasi-isomorphism of dg operads. 
Hence so is the composition 
\begin{equation}
\label{ti-F}
\ti{F} : = F \circ \ti{T} ~:~  \Cobar(P^{\ds}) \otimes \bbK ~\to~ \cO \otimes \bbK.
\end{equation}

Again, since diagram \eqref{diag-ti-T} commutes up to homotopy, we have
\begin{equation}
\label{H-ti-T}
H^{\bul} (\ti{T}) = H^{\bul}(\rho)^{-1} \circ T \circ H^{\bul}(\rho).
\end{equation}

Combining \eqref{H-F-rho-T} with \eqref{H-ti-T}, we deduce that 
$$
H^{\bul}(\ti{F}) =  H^{\bul} (F) \circ H^{\bul}(\ti{T})  =  
T^{-1} \circ  H^{\bul}(\rho) \circ 
H^{\bul}(\rho)^{-1} \circ T \circ H^{\bul}(\rho) = H^{\bul}(\rho). 
$$
In other words, both $\ti{F}$ and $\rho$ induce the same map at the level of cohomology. 

Let us denote by $\ti{\beta}$ the MC element in $\Conv(P, \cO \otimes \bbK)$
corresponding to the map $\ti{F}$.  

Since the diagram \eqref{diag-al-rho} for $\al$ commutes and the maps 
$\ti{F}$, $\rho$ induce the same map at the level of cohomology, we have
$$
\pi_{\cH} \circ (\ti{\beta} - \al) \big|_{\cG^1 P}  ~ = ~ 0, 
$$
where $\pi_{\cH}$ is the canonical projection from $\cZ(\cO) \to \cH$. 

Hence, composing $(\ti{\beta} - \al) \big|_{\cG^1 P}$ with a splitting \eqref{ti-ms}, we get 
a degree $0$ map of collections 
\begin{equation}
\label{h-cG1}
h : =  \ti{\ms} \circ (\ti{\beta} - \al)~:~ \cG^1 P \otimes \bbK \to \cO \otimes \bbK
\end{equation}
such that 
$$
\ti{\beta} (X) - \al (X) = \pa_{\cO} \circ  h (X) \qquad \forall~~ X \in \cG^1 P
$$
or equivalently\footnote{Recall that $\pa_P \big|_{\cG^1 P}  = 0$.}
\begin{equation}
\label{ti-beta-al}
\ti{\beta}(X) - \al (X) = \pa_{\cO} \circ  h (X) + h \circ \pa_P (X)  \qquad \forall~~ X \in \cG^1 P. 
\end{equation}

Let us extend $h$ to the degree zero element in $\Conv(P, \cO \otimes \bbK)$ by setting
$$
h \big|_{\cG^{\,> 1} P } = 0,
$$
and form the new MC element of $\Conv(P, \cO \otimes \bbK)$
\begin{equation}
\label{beta}
\beta : =   \exp([h, ~]) \ti{\beta}  ~  - ~ \frac{\exp([h, ~]) - 1}{[h, ~]} \, \pa h,
\end{equation}
where $\pa$ is the differential on $\Conv(P, \cO \otimes \bbK)$.

Equation \eqref{ti-beta-al} and Condition {\bf C1} imply that equation 
\eqref{cG-1-all-good} holds.

So the desired statement is proved. 
\end{proof}

Note that Proposition \ref{prop:beta-needed} already implies the statement of 
Lemma \ref{lem:betaalter} for $n=2$. So we can now assume that $n \ge 3$.
For this case, Lemma \ref{lem:betaalter} is a consequence of the following statement.
\begin{prop}
\label{prop:the-step}
Let $n > m \ge 2$ be integers and
$$
\al =  \sum_{k=1}^{n} \al_k\,, \qquad \al_k \in \Hom_{\Coll}(\cG^k P, \cO)
$$
be an $n$-th MC-sprout in $\Conv(P, \cO)$ for which the diagram \eqref{diag-al-rho} commutes. 
Furthermore, let 
$$
\beta = \sum_{k=1}^{\infty} \beta_k  \qquad \beta_k \in \Hom_{\Coll}(\cG^k P, \cO \otimes \bbK)  
$$
be a genuine MC element in $\Conv(P, \cO \otimes \bbK)$ such that 
\begin{equation}
\label{k-less-than-m}
\beta_k = \al_k \qquad \forall~~1 \le k \le m-1.
\end{equation}
Then there exists a MC element 
$$
\ti{\beta} = \sum_{k=1}^{\infty} \ti{\beta}_k  \qquad \ti{\beta}_k \in \Hom_{\Coll}(\cG^k P, \cO \otimes \bbK)  
$$
of $\Conv(P, \cO \otimes \bbK)$ such that  $\ti{\beta}_k = \al_k$ for every $k \le m$.
\end{prop}
%

\subsection{The sub-spaces $\Der^{(t)} \subset \Der \big(\Cobar(P^{\ds}) \big)$}
\label{sec:Der-prime}

Let us recall that, as the operad in the category of graded 
vector spaces\footnote{In this subsection, we assume that the base field is 
any field of characteristic zero.},  $\Cobar(P^{\ds})$ is the free operad generated 
by the collection $\bs P$. 
So, using the grading on the dg pseudo-operad $P$, we introduce the following 
grading on $\Cobar(P^{\ds})$: 
\begin{equation}
\label{weights-Cobar}
\Cobar(P^{\ds}) = \bigoplus_{q \ge 0} \Cobar(P^{\ds})^{(q)},    
\end{equation}
where $\Cobar(P^{\ds})^{(q)}$ is spanned by operadic monomials in 
${\bf s} X_1 \in \bs \cG^{k_1} P, ~ {\bf s} X_2 \in \bs \cG^{k_2} P, ~ \dots $ 
such that 
$$
\sum_{i \ge 1} (k_i -1)  = q.
$$
For example,  $\Cobar(P^{\ds})^{(0)}$ is precisely $\Op (\bs \cG^1 P)$ and 
$\Cobar(P^{\ds})^{(1)}$ is spanned by operadic monomials in 
$\Op (\bs \cG^1 P \oplus \bs \cG^2 P)$ for which a vector in $\bs \cG^2 P$ 
appears exactly once.

This grading is clearly compatible with the operadic multiplications on $\Cobar(P^{\ds})$.
In addition, Conditions \textbf{C1} and \textbf{C3} imply that
\begin{equation}
\label{diff-weight}
\pa \big( \Cobar(P^{\ds})^{(q)} \big) ~\subset~   \Cobar(P^{\ds})^{(q-1)} \qquad \forall ~~ q \ge 0,
\end{equation}
\begin{equation}
\label{rho-weight}
\rho\big|_{ \Cobar(P^{\ds})^{(q)} } = 0 \qquad \forall~~ q \ge 1, 
\end{equation}
and the map 
\begin{equation}
\label{rho-surj}
\rho\big|_{ \Cobar(P^{\ds})^{(0)} } ~:~  \Cobar(P^{\ds})^{(0)}  ~\to~ \cH 
\end{equation}
is onto. 

We claim that 
\begin{claim}
\label{cl:h-q-exist}
There exist maps of collections (for $q \ge 1$) of degree $-1$
\begin{equation}
\label{h-q}
h_q : \cZ\big( \Cobar(P^{\ds})^{(q)} \big) \to \Cobar(P^{\ds})^{(q+1)}
\end{equation}
and a degree $-1$ map of collections 
\begin{equation}
\label{h-0}
h_0 : \ker \big( \Cobar(P^{\ds})^{(0)} ~\stackrel{\rho}{\longrightarrow} ~\cH  \big) \to \Cobar(P^{\ds})^{(1)}
\end{equation}
such that 
$$
\pa \circ h_q (Y) = Y \qquad \forall~~ Y \in \cZ\big( \Cobar(P^{\ds})^{(q)} \big), ~~ q \ge 1, 
$$
$$
\pa \circ h_0 (Y) = Y \qquad \forall~~ Y \in  \ker \big( \Cobar(P^{\ds})^{(0)} ~\stackrel{\rho}{\longrightarrow} ~\cH  \big). 
$$
\end{claim}
\begin{proofOld}
Since $\rho$ is a quasi-isomorphism, the existence of the desired maps 
follows from \eqref{diff-weight}, \eqref{rho-weight}, \eqref{rho-surj} and 
the fact that we work with collections of vector spaces over a field of characteristic zero. 
\end{proofOld}

\vspace{0.18cm}

Let us denote by $\al_{\id}$ the MC element of $\Conv(P, \Cobar(P^{\ds}))$ corresponding to 
$$
\id :  \Cobar(P^{\ds}) \to  \Cobar(P^{\ds})
$$ 
and consider 
\begin{equation}
\label{Conv-P-Cobar}
\Conv(P, \Cobar(P^{\ds}))
\end{equation}
as the cochain complex with the differential $\pa + \pa_P + [\al_{\id}, ~]$, 
where $\pa$ (resp. $\pa_P$) is the differential coming from the one on 
$\Cobar(P^{\ds})$ (resp. $P$). 

Since $\Cobar(P^{\ds})$ is freely generated by $\bs P$, the assignment 
$$
\cD \mapsto \cD  \big|_{\bs P} ~\circ~ \bs
$$
gives us an isomorphism of graded vector spaces 
\begin{equation}
\label{Der-Conv}
\Der \big(\Cobar(P^{\ds}) \big) ~\cong~ \Conv(P, \Cobar(P^{\ds}))
\end{equation}
with the obvious shift: every degree $d$ derivation $\cD$ corresponds to a degree $d+1$
vector in $\Conv(P, \Cobar(P^{\ds}))$. 

Using the grading \eqref{weights-Cobar} on  $\Cobar(P^{\ds})$, we 
introduce the following subspaces of \eqref{Conv-P-Cobar} for $t \in \bbZ$ 
\begin{equation}
\label{Conv-P-Cobar-grading}
\cL^{(t)} : =
\big\{ f \in  \Conv(P, \Cobar(P^{\ds})) ~\big|~
f(\cG^q P) \subset \Cobar(P^{\ds})^{(q-1) + t}, ~~~ \forall ~ q \ge 1  \big\}. 
\end{equation}

Let us denote by $\{ \Der^{(t)} \}_{t \in \bbZ}$
the corresponding subspaces in $\Der \big(\Cobar(P^{\ds}) \big)$, i.e. 
\begin{equation}
\label{Der-t} 
\Der^{(t)}  : = \big\{ \cD \in \Der \big(\Cobar(P^{\ds}) \big) ~\big|~  \cD \big|_{\bs P} \circ \bs \in \cL^{(t)} \big\}. 
\end{equation}

It is clear that the commutator $[~, ~]$ on $\Der \big(\Cobar(P^{\ds}) \big)$ satisfies 
\begin{equation}
\label{brack-with-wghts}
[~,~] \, : \,  \Der^{(t_1)} \otimes \Der^{(t_2)} \subset \Der^{(t_1 + t_2)} \qquad \forall~~ t_1, t_2 \in \bbZ.
\end{equation}
Moreover, due to \eqref{diff-weight}
\begin{equation}
\label{diff-weight-Der}
[\pa, ~] :  \Der^{(t)}  \to \Der^{(t-1)} \qquad \forall ~~ t \in \bbZ,
\end{equation}
where $\pa$ is the full differential on $\Cobar(P^{\ds})$. 

Let us prove the following statement 
\begin{claim}
\label{cl:loc-nilpot}
Let $t$ be a negative integer and $\cD$ be a degree $0$ derivation 
in $ \Der^{(t)}$. Then $\cD$ acts locally nilpotently on $\Cobar(P^{\ds})$, i.e. 
for every $X \in \Cobar(P^{\ds})$, there exists an integer $m$ such that 
$$
\cD^{m}(X) = 0. 
$$
\end{claim}
\begin{proofOld}
Since every vector in $\Cobar(P^{\ds})$ is a finite linear combination of operadic monomials
in ${\bf s} P$, it suffices to prove that for every $X \in {\bf s} P$, there exists $m$ such that
$$
\cD^m (X) = 0. 
$$

Again, since every $X \in {\bf s} P$ is a linear combination of vectors in $\bs \cG^{k} P$ for 
various $k$'s, we may assume without loss of generality, that  $X \in \bs \cG^k P$ for some $k \ge 1$.  

By definition of $\Der^{(t)}$, we have
$$
\cD^m (X) \subset \Cobar(P^{\ds})^{((k -1) + m t )}. 
$$
So the desired statement follows from the fact that 
$$
\Cobar(P^{\ds})^{(r )} = \bfzero \qquad \forall ~~ r < 0.  
$$
\end{proofOld}

Claim \ref{cl:loc-nilpot} implies that
\begin{claim}
\label{cl:exp-OK}
For every negative integer $t$, every $\pa$-closed degree degree $0$ derivation 
$$
\cD \in  \Der^{(t)} \subset \Der \big(\Cobar(P^{\ds}) \big)
$$
gives us the automorphism of the dg operad $\Cobar(P^{\ds})$
$$
\exp(\cD) : \Cobar(P^{\ds}) \stackrel{\cong}{\longrightarrow} \Cobar(P^{\ds}).  
$$
\end{claim}
\begin{proofOld}
Claim \ref{cl:loc-nilpot} implies that the formal Taylor series
$$
\exp(\cD) : = \id + \sum_{k \ge 1} \frac{1}{k!} \cD^k  
$$
is a well defined automorphism of the graded operad $\Op({\bf s} P)$. 

Since $\cD$ is $\pa$-closed, this automorphism is also 
compatible with the differential on $\Cobar(P^{\ds})$. 
\end{proofOld}

Let us now prove the following technical statement:
\begin{prop}
\label{prop:psi-lift}
Let $m$ be an integer $\ge 2$ and 
$$
\psi \in \Hom_{\Coll} (\cG^m P, \cH) \in  \Conv(P,  \cH)
$$
be a degree $1$ vector satisfying 
\begin{equation}
\label{psi-closed}
\psi \circ \pa_{P} + [\al_{\cH}, \psi] =0. 
\end{equation}
Then there exists a degree $0$ $\pa$-closed derivation 
$\cD \in  \Der^{(1-m)} \subset \Der \big(\Cobar(P^{\ds}) \big)$ such that
\begin{equation}
\label{cD-psi}
\rho \circ \cD \circ \bs\, \big|_{P}  = \psi  
\end{equation}
and
\begin{equation}
\label{cD-less-m}
\cD (\bs X) = 0 \qquad \forall ~~ X \in \cG^{< m} P.
\end{equation}
\end{prop}
\begin{proofOld} Since 
$$
\rho \big|_{\Op(\bs \cG^1 P)} : \Op(\bs  \cG^1 P)  \to \cH
$$
is onto (and we work with fields of characteristic zero), there exists 
a degree $1$ vector 
$$
\Psi_m \in \Hom_{\Coll}\big(\cG^m P,   \Op(\bs \cG^1 P) \big) \subset  \Conv \big(P, \Cobar(P^{\ds}) \big)
$$ 
such that $\rho \circ \Psi_m(X) = \psi(X)$ for all $X \in \cG^m P$. 
Clearly, $\Psi_m \in \cL^{(1-m)}$ and $\Psi_m$ satisfies the equation 
$$
\pa \Psi_m (X) + \Psi_m (\pa_P X) + [\al_{\id}, \Psi_m] (X) = 0 \qquad \forall ~~ X \in \cG^{\le m} P.  
$$

Due to \eqref{psi-closed}, the map 
$$
\big( \Psi_m \circ \pa_P + [\al_{\id}, \Psi_m]  \big) \big|_{ \cG^{m+1} P  } :  \cG^{m+1} P \to  \Cobar(P^{\ds})^{(0)} 
$$
lands in the kernel of $\rho$. 

Hence, by Claim \ref{cl:h-q-exist}, the map 
$$
\Psi_{m+1} (Y) : = - h_0 \big( \Psi_m (\pa_P Y) + [\al_{\id}, \Psi_m] (Y)\big) ~:~   \cG^{m+1} P  \to  \Cobar(P^{\ds})^{(1)} 
$$
satisfies
$$
\pa \Psi_{m+1} (Y) + \Psi_m (\pa_P Y) + [\al_{\id}, \Psi_m] (Y) = 0 \qquad \forall ~~Y \in \cG^{(m+1)} P.  
$$ 

Therefore, the sum $\Psi^{(m+1)} = \Psi_m + \Psi_{m+1} $ satisfies the equation 
$$
\pa \Psi^{(m+1)} (Y) + \Psi^{(m+1)}  (\pa_P Y) + [\al_{\id}, \Psi^{(m+1)} ] (Y) = 0 \qquad \forall ~~Y \in \cG^{\le (m+1)} P.  
$$ 
Moreover, $\Psi^{(m+1)}$ belongs to $\cL^{(1-m)}$ by construction.   

Let us assume that we extended $\Psi^{(m+1)}$ to a vector (for some $k \ge 1$)
$$
\Psi^{(m+k)} = \Psi_m +  \Psi_{m+1}  + \dots +  \Psi_{m+k}\,, \qquad \Psi_j \in  \Hom_{\Coll}\big(\cG^j P,    \Cobar(P^{\ds})^{(j-m)}  \big)
$$
such that 
\begin{equation}
\label{Psi-mk-closed}
\pa \Psi^{(m+k)} (Y) + \Psi^{(m+k)}  (\pa_P Y) + [\al_{\id}, \Psi^{(m+k)} ] (Y) = 0 \qquad \forall ~~Y \in \cG^{\le (m+k)} P.  
\end{equation}

Let $X \in \cG^{m+k+1} P$. Using \eqref{Psi-mk-closed} and the MC equation 
$$
\pa \circ \al_{\id} + \al_{\id}\circ \pa_P + \frac{1}{2} [\al_{\id}, \al_{\id}] = 0
$$
for $\al_{\id}$, we deduce that 
$$
\pa \big( \Psi^{(m+k)}  (\pa_P X) + [\al_{\id}, \Psi^{(m+k)} ] (X) \big)  =  
 - [\al_{\id}, \Psi^{(m+k)} ] (\pa_P X)  + \pa \big( [\al_{\id}, \Psi^{(m+k)} ] (X) \big)
$$
$$
= [\pa \circ \al_{\id} + \al_{\id} \circ \pa_P\,, \, \Psi^{(m+k)} ] (X) - 
[\al_{\id}\,,\, \pa \circ  \Psi^{(m+k)} +  \Psi^{(m+k)} \circ \pa_P ] (X)  
$$
$$
= \big( [\al_{\id}, [\al_{\id},   \Psi^{(m+k)} ]] - \frac{1}{2} [[\al_{\id}, \al_{\id}],   \Psi^{(m+k)}] \big)(X) = 0. 
$$
In other words, the map
$$
\big( \Psi^{(m+k)} \circ \pa_P + [\al_{\id}, \Psi^{(m+k)} ]  \big) \big|_{ \cG^{m+k+1} P} ~:~   \cG^{m+k+1} P 
~\to~  \Cobar(P^{\ds})^{(k)}
$$
lands in $\cZ(\Cobar(P^{\ds})^{(k)})$. 

Hence, by Claim \ref{cl:h-q-exist}, the map 
$$
\Psi_{m+k+1} (X) : = - h_k \big( \Psi^{(m+k)} (\pa_P X) + [\al_{\id}, \Psi^{(m+k)}] (X)\big) ~:~   \cG^{m+k+1} P  \to  \Cobar(P^{\ds})^{(k+1)} 
$$
satisfies the equation 
\begin{equation}
\label{next}
\pa \Psi_{m+k+1} (X) + \Psi^{(m+k)}  (\pa_P X) + [\al_{\id}, \Psi^{(m+k)} ] (X) = 0 \qquad \forall ~~X \in   \cG^{m+k+1} P. 
\end{equation}
Therefore the vector 
$$
 \Psi^{(m+k+1)}  : = \Psi^{(m+k)} +  \Psi_{m+k+1} =  \Psi_{m} + \Psi_{m+1} + \dots +  \Psi_{m+k+1}
$$
satisfies the equation 
\begin{equation}
\label{Psi-mk1-closed}
\pa \Psi^{(m+k+1)} (X) + \Psi^{(m+k+1)}  (\pa_P X) + [\al_{\id},  \Psi^{(m+k+1)} ] (X) = 0 \qquad \forall ~~X \in   \cG^{\le (m+k+1)} P. 
\end{equation}
Moreover, since $ \Psi_{m+k+1} \in \cL^{(1-m)}$, the vector $\Psi^{(m+k+1)}$ also belongs to $\cL^{(1-m)}$. 

This inductive argument shows that there exists a degree $1$ vector 
$$
\Psi  = \sum_{j=m}^{\infty} \Psi_j\,, \qquad \Psi_j \in  \Hom_{\Coll}\big(\cG^j P,    \Cobar(P^{\ds})^{(j-m)}  \big)
$$
such that 
\begin{equation}
\label{Psi-closed}
\pa \circ \Psi + \Psi \circ \pa_P + [\al_{\id}, \Psi] = 0
\end{equation}
and 
\begin{equation}
\label{rho-Psi-m}
\rho \circ \Psi_m = \psi. 
\end{equation}

Since $\rho(Z) = 0$ for every $Z \in  \Cobar(P^{\ds})^{(t)} $ if $t \ge 1$,  equation \eqref{rho-Psi-m}
implies that 
\begin{equation}
\label{rho-Psi}
\rho \circ \Psi = \psi. 
\end{equation}

Equation \eqref{Psi-closed} implies that the (degree $0$) derivation 
$$
\cD \in \Der^{(1-m)} \subset \Der\big( \Cobar(P^{\ds}) \big)
$$ 
corresponding to $\Psi$ is $\pa$-closed. Furthermore, equation \eqref{rho-Psi} implies \eqref{cD-psi}. 
Finally, equation \eqref{cD-less-m} is a consequence of 
$$
\Psi  \big|_{\cG^{< m } P} ~ = ~ 0. 
$$
\end{proofOld}

\subsection{The proof of Proposition \ref{prop:the-step}}
\label{sec:proof}
We will now use Proposition \ref{prop:psi-lift} to prove Proposition \ref{prop:the-step}.

Since $\al$ is an $n$-th MC-sprout and $\beta$ is a genuine MC element of $\Conv(P, \cO \otimes \bbK)$, 
we have
\begin{equation}
\label{curv-beta}
\pa_{\cO} \circ \beta_m + \beta_{m-1} \circ \pa_P + \sum_{k=1}^{m-1} \beta_k \bullet \beta_{m-k} = 0,
\end{equation}
\begin{equation}
\label{curv-al}
\pa_{\cO} \circ \al_m + \al_{m-1} \circ \pa_P + \sum_{k=1}^{m-1} \al_k \bullet \al_{m-k} = 0,
\end{equation}
\begin{equation}
\label{curv-beta-next}
\pa_{\cO} \circ \beta_{m+1} + \beta_{m} \circ \pa_P + [\beta_1, \beta_m]  +  \sum_{k=2}^{m-1} \beta_k \bullet \beta_{m+1-k} = 0,
\end{equation}
and
\begin{equation}
\label{curv-al-next}
\pa_{\cO} \circ \al_{m+1} + \al_{m} \circ \pa_P + [\al_1, \al_m]  +  \sum_{k=2}^{m-1} \al_k \bullet \al_{m+1-k} = 0.
\end{equation}

Subtracting \eqref{curv-al} from \eqref{curv-beta} and using \eqref{k-less-than-m}, we get
$$
\pa_{\cO} \circ (\beta_m - \al_m) = 0.
$$
In other, words $\beta_m - \al_m$ is a map from $\cG^m P $ to $\cZ(\cO  \otimes \bbK)$. 

Let 
\begin{equation}
\label{psi-m}
\psi_m : = \pi_{\cH} \circ (\beta_m - \al_m) \in \Conv(P, \cH \otimes \bbK).
\end{equation}

Subtracting \eqref{curv-al-next} from \eqref{curv-beta-next} and using \eqref{k-less-than-m} again, we get
\begin{equation}
\label{m-plus-1}
(\beta_{m} - \al_m) \circ \pa_P + [\al_1, \beta_m- \al_m] = -\pa_{\cO} \circ (\beta_{m+1}  - \al_{m+1}).
\end{equation}

Next, we observe that both sides of \eqref{m-plus-1} are maps which land in $\cZ(\cO \otimes \bbK)$. So applying 
$\pi_{\cH}$ to both sides of  \eqref{m-plus-1} and using $\pi_{\cH} \circ \al_1 = \al^{\cH}$, we deduce that 
$$
\psi_m \circ \pa_P +  [\al_{\cH}, \psi_m] = 0. 
$$ 
In other words, $\psi_m$ is a cocycle in the cochain complex
$$
\Conv(P, \cH \otimes \bbK)
$$
with the differential $\pa_P + [\al_{\cH}, ~]$. 

Due to Proposition \ref{prop:psi-lift}, there exists a $\pa$-closed degree zero derivation 
$$
\cD \in \Der^{(1-m)} \subset \Der \big( \Cobar(P^{\ds}) \otimes \bbK \big) 
$$
such that 
\begin{equation}
\label{cD-psi-m}
\rho \circ \cD \circ \bs\, \big|_{P}  = \psi_m  
\end{equation}
and
\begin{equation}
\label{cD-less-m-here}
\cD (\bs X) = 0 \qquad \forall ~~ X \in \cG^{< m} P.
\end{equation}

Thanks to Claim \ref{cl:exp-OK}, $-\cD$ can be exponentiated to the automorphism 
$\exp(-\cD)$ of the dg operad $\Cobar(P^{\ds}) \otimes \bbK$.

Let $F_{\beta}$ be the quasi-isomorphism of dg operads $ \Cobar(P^{\ds}) \otimes \bbK \to \cO \otimes \bbK$ corresponding 
to the MC element $\beta$.  Due to \eqref{cD-less-m-here}, the quasi-isomorphism 
$$
F : = F_{\beta} \circ \exp(-\cD) ~:~  \Cobar(P^{\ds}) \otimes \bbK \to \cO \otimes \bbK
$$
satisfies 
$$
F_{\beta} \circ \exp(-\cD)(\bs X) = F_{\beta}(\bs X) \qquad \forall ~~ X \in \cG^{< m} P. 
$$
Furthermore,
$$
F_{\beta} \circ \exp(-\cD)(\bs X) -  F_{\beta}(\bs X) \in \cZ(\cO \otimes \bbK) \qquad \forall~~ X \in \cG^{m} P.
$$

Using equations $\pi_{\cH} \circ \beta_1 = \al^{\cH}$ and \eqref{cD-psi-m}, we deduce that 
$$
\pi_{\cH} \big( F_{\beta} \circ \exp(-\cD)(\bs X) -  F_{\beta}(\bs X) \big) = - \psi_m (X).
$$

Thus the MC element 
$$
\beta^{\dia} =  \sum_{k=1}^{\infty} \beta^{\dia}_k  \qquad \beta^{\dia}_k \in \Hom_{\Coll}(\cG^kP, \cO \otimes \bbK) 
$$ 
corresponding to $F$ has these properties:
$$
\beta^{\dia}_k  = \beta_k (= \al_k) \qquad \forall~~  k < m, 
$$
$$
(\beta^{\dia}_m - \beta_m) (X) \in \cZ(\cO) \qquad \forall~~ X \in \cG^m P
$$
and
$$
\pi_{\cH} \circ (\beta^{\dia}_m - \beta_m) (X)  =  - \psi_m \qquad \forall~~ X \in \cG^m P
$$ 
or equivalently 
\begin{equation}
\label{beta-dia-al}
\pi_{\cH} \circ (\beta^{\dia}_m - \al_m) (X)  = 0 \qquad \forall~~ X \in \cG^m P. 
\end{equation}

Hence, using the splitting \eqref{ti-ms}, we define the following degree $0$ vector 
$$
\xi \in  \Hom_{\Coll}(\cG^m  P, \cO \otimes \bbK) 
$$
\begin{equation}
\label{xi-dfn}
\xi(X) : = \ti{\ms} \circ (\beta^{\dia}_m - \al_m) (X) \qquad X \in \cG^m P,
\end{equation}
which satisfies 
\begin{equation}
\label{xi-beta-al}
\beta^{\dia}_m(X) = \al_m (X) + \pa_{\cO} \circ  \xi (X).  
\end{equation}

The desired MC element $\ti{\beta}$ is defined by the formula
$$
\ti{\beta} =  \exp([\xi, ~]) \beta^{\dia}  ~  - ~ \frac{\exp([\xi, ~]) - 1}{[\xi, ~]} \, \pa \xi.
$$

Indeed, since $\xi(X) = 0$ for all $X \in \cG^{< m} P$,
$$
\ti{\beta}_k = \beta_k = \al_k \qquad \forall~~ k < m.  
$$   
Moreover, equation \eqref{xi-beta-al} implies that
$$
\ti{\beta}_m = \al_m\,.
$$

Thus Proposition \ref{prop:the-step} is proved. $\qed$

\appendix

\section{The lifting property for cobar resolutions}
\label{app:lift}
Let us recall that the functor $\Conv(P, ?)$ preserves quasi-isomorphisms:
\begin{prop}
\label{prop:Conv-q-iso}
If $P$ is a dg pseudo-operad satisfying Condition \ref{cond:P-filtered} and 
$f : \cA \to \cB$ is a quasi-isomorphism of dg operads, then the restriction of 
$$
f_* : \Conv(P, \cA) \to \Conv(P, \cB)
$$
to $\cF_m \Conv(P, \cA)$ is a quasi-isomorphism of dg Lie algebras 
\begin{equation}
\label{f-for-cF-m}
f_* \big|_{\cF_m \Conv(P, \cA)} ~:~ \cF_m \Conv(P, \cA) ~\to~ \cF_m \Conv(P, \cB)
\end{equation}
for every $m \ge 1$. 
\end{prop}
\begin{proofOld}
This statement was proved in \cite[Section 4.4]{notes} for non-colored 
operads under the assumption that $P$ has the zero differential. 
Here we will give the proof of the more general statement. 

Since $f_*$ is compatible with the Lie brackets, we may forget about the Lie brackets
on $\cF_m \Conv(P, \cA)$ and $\cF_m \Conv(P, \cB)$ and treat both the source and the 
target of \eqref{f-for-cF-m} as cochain complexes with the differentials coming from those 
on $P$, $\cA$, and $\cB$. 

Since we deal with cochain complexes of vector spaces, there exists degree zero maps 
\begin{equation}
\label{ti-g-bq}
\ti{g}_{\bq} : \cB(\bq) \to \cA(\bq)
\end{equation}
and degree $-1$ maps 
\begin{equation}
\label{ti-chi-bq-cA}
\ti{\chi}_{\bq, \cA} : \cA(\bq) \to \cA(\bq),
\end{equation}
\begin{equation}
\label{ti-chi-bq-cB}
\ti{\chi}_{\bq, \cB} : \cB(\bq) \to \cB(\bq),
\end{equation}
such that 
\begin{equation}
\label{homotopy-cB}
f \circ \ti{g}_{\bq} (v)  = v + \pa_{\cB} \circ \ti{\chi}_{\bq, \cB} (v) + \ti{\chi}_{\bq, \cB} \circ \pa_{\cB} (v), 
\qquad \forall~~ v \in \cB(\bq)
\end{equation}
and
\begin{equation}
\label{homotopy-cA}
\ti{g}_{\bq} \circ f (w)  = w + \pa_{\cA} \circ \ti{\chi}_{\bq, \cA} (w) + \ti{\chi}_{\bq, \cA} \circ \pa_{\cA} (w), 
\qquad \forall~~ w \in \cA(\bq),
\end{equation}
where $\bq$ is any $\Xi$-colored corolla and $\pa_{\cA}$ (resp. $\pa_{\cB}$) is the differential 
on $\cA$ (resp. on $\cB$).

Moreover, since our base field has characteristic zero, we can use the 
standard averaging operators (for products of symmetric groups) and turn 
the maps \eqref{ti-g-bq}, \eqref{ti-chi-bq-cA}, and \eqref{ti-chi-bq-cB} into maps
of collections
\begin{equation}
\label{g-cB-cA}
g : \cB \to \cA, \qquad 
\chi_{\cA} : \cA \to \cA, \qquad 
\chi_{\cB} : \cB \to \cB, 
\end{equation}
satisfying
\begin{equation}
\label{homotopy}
g \circ f = \id_{\cA} + \pa_{\cA}  \circ \chi_{\cA}  + \chi_{\cA} \circ \pa_{\cA}, 
\qquad 
f \circ g = \id_{\cB} + \pa_{\cB}  \circ \chi_{\cB}  + \chi_{\cB} \circ \pa_{\cB}.
\end{equation}

Inclusion \eqref{diff-P-filtr} guarantees that if $f \in \cF_m \Conv(P, \cO)$ 
(for any dg operad $\cO$) then 
$$
f \circ \pa_P \in \cF_{m+1} \Conv(P, \cO).
$$
Hence the differential on the associated graded complex 
\begin{equation}
\label{Gr-Conv}
\bigoplus_{k \ge m} \cF_k \Conv(P, \cO) \big/ \cF_{k+1} \Conv(P, \cO)
\end{equation}
comes solely from the differential $\pa_{\cO}$ on $\cO$.

Therefore, equations in \eqref{homotopy} imply that the map \eqref{f-for-cF-m}
induces a quasi-isomorphism for the associated graded complexes: 
$$
\bigoplus_{k \ge m} \cF_k \Conv(P, \cA) \big/ \cF_{k+1} \Conv(P, \cA) ~\stackrel{\sim}{\longrightarrow}~
\bigoplus_{k \ge m} \cF_k \Conv(P, \cB) \big/ \cF_{k+1} \Conv(P, \cB). 
$$ 

Thus, since $\cF_m \Conv(P, \cA)$ (resp. $\cF_m \Conv(P, \cB)$) is complete with respect to 
the filtration 
$\cF_m \Conv(P, \cA) \supset \cF_{m+1} \Conv(P, \cA) \supset \dots $
(resp. $\cF_m \Conv(P, \cB) \supset \cF_{m+1} \Conv(P, \cB) \supset \dots $), the map 
\eqref{f-for-cF-m} is indeed a quasi-isomorphism. (See, for example, Lemma D.1 from 
\cite{DeligneTW}).  
\end{proofOld}

Proposition \ref{prop:Conv-q-iso} has the following corollaries:
\begin{cor}
\label{cor:lift}
Let $\Psi: \cA \to \cB$ be a quasi-isomorphism of
dg operads and $P$ be a dg pseudo-operad satisfying 
Condition \ref{cond:P-filtered}. Then for every operad map 
$R_{\cB} : \Cobar(P^{\ds}) \to \cB$ there exists an operad 
map $R_{\cA} : \Cobar(P^{\ds}) \to \cA$ such that the diagram 
\begin{equation}
\label{Cobar-cA-cB}
\begin{tikzpicture}
\matrix (m) [matrix of math nodes, row sep=2em, column sep=2em]
{~~~ & \Cobar(P^{\ds})  ~ \\
 \cA &  \cB\\ };
\path[->, font=\scriptsize]
(m-1-2) edge node[right] {$R_{\cB}$} (m-2-2)  
(m-2-1) edge  node[auto] {$\Psi$} (m-2-2);
\path[->, dashed, font=\scriptsize] 
(m-1-2) edge node[above] {$R_{\cA}~~$} (m-2-1);
\end{tikzpicture}
\end{equation}
commutes up to homotopy. 
Moreover, if $\ti{R}_{\cA}$ is another operad map $\Cobar(P^{\ds}) \to \cA$
for which $\Psi \circ \ti{R}_{\cA}$ is homotopy equivalent 
to $R_{\cB}$ then  $\ti{R}_{\cA}$ is homotopy equivalent to $R_{\cA}$.
\end{cor}
\begin{cor}
\label{cor:zig-zag-shorter}
Let $\cO$ be a dg operad defined over a field $\bbK$ of characteristic zero
and $\Cobar(P^{\ds})$ be a cobar resolution of another dg operad 
$\wt{\cO}$, where 
$P$ is a dg pseudo-cooperad satisfying Condition \ref{cond:P-filtered}.
Then the existence of a zig-zag of quasi-isomorphisms (of dg operads) 
\begin{equation}
\label{zig-zag}
\cO \,\stackrel{\sim}{\leftarrow}\, \bullet  \,\stackrel{\sim}{\rightarrow}\, \bullet
\,\stackrel{\sim}{\leftarrow}\, \bullet ~ \dots ~ \bullet  \,\stackrel{\sim}{\rightarrow}\, \wt{\cO}
\end{equation}
is equivalent to the existence of a quasi-isomorphism of dg operads
\begin{equation}
\label{from-cobar}
F :  \Cobar(P^{\ds})  \,\stackrel{\sim}{\rightarrow}\,  \cO. 
\end{equation}
\end{cor} 
\begin{proof}[Corollary \ref{cor:lift}]
Due to Proposition \ref{prop:Conv-q-iso}, the map 
$$
\Psi_* : \Conv(P, \cA)  \,\stackrel{\sim}{\rightarrow}\, \Conv(P, \cB)
$$ 
induced by the quasi-isomorphism  $\Psi: \cA \to \cB$ is a quasi-isomorphism 
of filtered dg Lie algebras satisfying the necessary conditions of 
\cite[Theorem 1.1]{GMtheorem}. 

Therefore, there exists a MC element 
$$
\al_{\cA} \in \Conv(P, \cA) 
$$
for which $\Psi_*(\al_{\cA})$ is equivalent to the MC element $\al_{\cB} \in \Conv(P, \cB)$
corresponding to the operad map $R_{\cB}$. 

Hence \cite[Theorem 5.6]{notes} implies\footnote{This theorem is 
proved in \cite{notes} only for non-colored operads. 
However, the proof can be easily generalized to the case of colored operads.} 
that $R_{\cB}$ is homotopy equivalent 
to $\Psi \circ R_{\cA}$, 
where $R_{\cA}$ is the operad map $\Cobar(P^{\ds}) \to \cA$ corresponding to 
the MC element $\al_{\cA}$.

Let $\ti{R}_{\cA}$ be another operad map $\Cobar(P^{\ds}) \to \cA$
for which $\Psi \circ \ti{R}_{\cA}$ is homotopy equivalent 
to $R_{\cB}$ and $\ti{\al}_{\cA}$ be the MC element of $\Conv(P, \cA)$
corresponding to $\ti{R}_{\cA}$. 

Since $\Psi \circ \ti{R}_{\cA}$ is homotopy equivalent to $R_{\cB}$, the 
MC elements $\Psi_*(\ti{\al}_{\cA})$ and $\al_{\cB}$ are isomorphic. 
Hence, applying Theorem \cite[Theorem 1.1]{GMtheorem} once again, we 
conclude that $\ti{\al}_{\cA}$ is isomorphic to $\al_{\cA}$.  

Thus  $\ti{R}_{\cA}$ is homotopy equivalent to $R_{\cA}$. 
\end{proof}

\vspace{0.18cm}

\begin{proof}[Corollary \ref{cor:zig-zag-shorter}]
Let us denote by $\rho$ the quasi-isomorphism
\begin{equation}
\label{the-rho}
\rho :  \Cobar(P^{\ds})  \,\stackrel{\sim}{\rightarrow}\,  \wt{\cO}. 
\end{equation}

Given a quasi-isomorphism $F$ in \eqref{from-cobar}, we produce the zig-zag
of quasi-isomorphisms of dg operads 
$$
\cO \,\stackrel{\, F}{\longleftarrow}\, \Cobar(\cC) \,\stackrel{\rho}{\longrightarrow}\, \wt{\cO}.
$$

So the implication $\Leftarrow$ is obvious. 

To proof of the implication $\Rightarrow$ is based on the obvious application of 
the lifting property from Corollary \ref{cor:lift} and the 2-out-of-3 property 
for quasi-isomorphisms. 
\end{proof} 

\section{Tamarkin's $\Ger_{\infty}$-structure up to arity $4$}
\label{app:Tam-Arity4}

In \cite{Software}, we developed a software which implements the recursive construction of 
a quasi-isomorphism $\Ger_{\infty} \to \Br$ over rationals.

Let us recall that an $n$-th sprout in $\Conv(\Ger^{\vee}, \Br)$ is identified with a degree $1$ vector in 
$$
\al \in \bigoplus_{m=2}^{n+1} \big( \Br(m) \otimes \La^{-2} \Ger(m) \big)_{S_m}
$$
for which 
$$
\pa \al + \frac{1}{2}[\al, \al]  ~\in~  \big( \Br(n+2) \otimes \La^{-2} \Ger(n+2) \big)_{S_{n+2}} ~\oplus~ 
\big( \Br(n+3) \otimes \La^{-2} \Ger(n+3) \big)_{S_{n+3}} ~\oplus~ \dots  
$$
In other words, $n$-th MC-sprout $\al$ involves terms in arities $2,3,\dots, n+1$ and all terms of 
$\Curv(\al)$ have arities $\ge n+2$. 

Using this software, we found a $4$-th MC-sprout $\al$. This sprout has 1265 terms and the truncation 
$\al^{[3]}$ of $\al$ is shown in figures \ref{fig:Part1}, \ref{fig:Part11}, and \ref{fig:lastPart}.
Due to Corollary \ref{cor:every}, there exists a genuine MC element $\al_{MC} \in \Conv(\Ger^{\vee}, \Br)$ 
such that 
$$
\al^{[3]} = \al^{[3]}_{MC}\,.
$$
In other words, the truncation $\al^{[3]}$ shown in figures \ref{fig:Part1}, \ref{fig:Part11}, and \ref{fig:lastPart}
can be extended to a genuine MC element in $\Conv(\Ger^{\vee}, \Br)$. 

We would like to remark that, if we subtract the blue terms (in figure \ref{fig:Part1}) from $\al^{[3]}$, then
the resulting degree $1$ element $\ti{\al}$ will still be a third MC-sprout. However, we 
proved\footnote{The verification of this fact on a modern MacAir using \cite{Software} took approximately 5 hours.}
that $\ti{\al}$ is not a truncation of any $4$-th MC-sprout in $\Conv(\Ger^{\vee}, \Br)$. So $\ti{\al}$ cannot be extended 
to a genuine MC element in $\Conv(\Ger^{\vee}, \Br)$.

%
%
\begin{figure}[htp] 
\centering 
\begin{tikzpicture}[scale=0.5]
\tikzstyle{lab}=[circle, draw, minimum size=5, inner sep=1]
\tikzstyle{n}=[circle, draw, fill, minimum size=5, inner sep=1]
\tikzstyle{root}=[circle, draw, fill, minimum size=0, inner sep=1]
\begin{scope}[shift={(0,5.5)}] 
\draw (-1.2,0) node[anchor=center] {{$ \frac{1}{2}$}};
\node [n] (n) at (0,0) {};
\node [lab] (v1) at (-0.4,0.6) {\tiny $1$};
\node [lab] (v2) at (0.4,0.6) {\tiny $2$};
\draw (n) edge (v1) edge (v2) edge (0, -0.4);
\draw (2.5,0) node[anchor=center] {{\small $\otimes ~ \{ b_1, b_2 \}$}};
\end{scope}
\begin{scope}[shift={(5.5,5.5)}]
\draw (-1,0) node[anchor=center] {{$ + $}};
\node [lab] (v1) at (0,0) {\tiny $1$};
\node [lab] (v2) at (0, 0.7) {\tiny $2$};
\draw (v1) edge (v2) edge (0, -0.5);
\draw (1.5,0) node[anchor=center] {{\small $\otimes ~ b_1 b_2 $}};
\end{scope}
\begin{scope}[shift={(10.5,5.5)}]
\draw (-1.5,0) node[anchor=center] {{$ + \frac{1}{2}$}};
\node [lab] (v1) at (0,0) {\tiny $1$};
\node [lab] (v2) at (-0.4, 0.6) {\tiny $2$};
\node [lab] (v3) at (0.4, 0.6) {\tiny $3$};
\draw (v1) edge (v2) edge (v3) edge (0, -0.5);
\draw (2.5,0) node[anchor=center] {{\small $\otimes ~ b_1 \{b_2, b_3\} $}};
\end{scope}
\begin{scope}[shift={(17,5.5)}]
\draw (-1.5,0) node[anchor=center] {{$ - \frac{1}{3}$}};
\node [n] (n) at (0,0) {};
\node [lab] (v1) at (-0.7, 0.7) {\tiny $1$};
\node [lab] (v2) at (0, 0.7) {\tiny $2$};
\node [lab] (v3) at (0.7, 0.7) {\tiny $3$};
\draw (n) edge (v1) edge (v2) edge (v3) edge (0, -0.5);
\draw (3,0) node[anchor=center] {{\small $\otimes ~\{b_1, \{b_2, b_3 \}\}$}};
\end{scope}
\begin{scope}[shift={(2,3)}]
\draw (-1.5,0) node[anchor=center] {{$ - \frac{1}{6}$}};
\node [n] (n) at (0,0) {};
\node [lab] (v1) at (-0.7, 0.7) {\tiny $1$};
\node [lab] (v2) at (0, 0.7) {\tiny $2$};
\node [lab] (v3) at (0.7, 0.7) {\tiny $3$};
\draw (n) edge (v1) edge (v2) edge (v3) edge (0, -0.5);
\draw (3,0) node[anchor=center] {{\small $\otimes ~\{b_2, \{b_1, b_3 \}\}$}};
\end{scope}
\begin{scope}[shift={(9.5,3)}]
\draw (-1.5,0) node[anchor=center] {{$ - \frac{1}{12}$}};
\node [n] (n) at (0,0) {};
\node [lab] (v1) at (-0.4, 0.6) {\tiny $1$};
\node [lab] (v3) at (0.4, 0.6) {\tiny $3$};
\node [lab] (v2) at (-0.8, 1.2) {\tiny $2$};
\draw (n) edge (v1) edge (v3) edge (0, -0.5) (v1) edge (v2);
\draw (3,0) node[anchor=center] {{\small $\otimes ~\{b_1, \{b_2, b_3 \}\}$}};
\end{scope}
\begin{scope}[shift={(17,3)}]
\draw (-1.5,0) node[anchor=center] {{$ - \frac{1}{12}$}};
\node [n] (n) at (0,0) {};
\node [lab] (v1) at (-0.4, 0.6) {\tiny $1$};
\node [lab] (v2) at (0.4, 0.6) {\tiny $2$};
\node [lab] (v3) at (0.8, 1.2) {\tiny $3$};
\draw (n) edge (v1) edge (v2) edge (0, -0.5) (v2) edge (v3);
\draw (3,0) node[anchor=center] {{\small $\otimes ~\{b_2, \{b_1, b_3 \}\}$}};
\end{scope}
%
\begin{scope}[shift={(0,0)}] 
\draw (-2,0.3) node[anchor=center] {{$- \frac{1}{3}$}};
\node [lab] (v1) at (0,0) {\tiny $1$};
\node [lab] (v2) at (-1,1) {\tiny $2$};
\node [lab] (v3) at (0,1) {\tiny $3$};
\node [lab] (v4) at (1,1) {\tiny $4$};
\draw (v1) edge (v2) edge (v3) edge (v4) edge (0,-0.7);
\draw (4,0.3) node[anchor=center] {{\small $\otimes ~ b_1 \{b_2, \{ b_3, b_4 \} \}$}};
\end{scope}
\begin{scope}[shift={(10,0)}]
\draw (-2,0.3) node[anchor=center] {{$- \frac{1}{12}$}};
\node [lab] (v1) at (0,0) {\tiny $1$};
\node [lab] (v2) at (-0.5,0.7) {\tiny $2$};
\node [lab] (v3) at (-1,1.4) {\tiny $3$};
\node [lab] (v4) at (0.5,0.7) {\tiny $4$};
\draw (v1) edge (v2) edge (v4) edge (0,-0.7) (v2) edge (v3);
\draw (4,0.3) node[anchor=center] {{\small $\otimes ~ b_1 \{b_2, \{ b_3, b_4 \} \}$}};
\end{scope}
\begin{scope}[shift={(20,0)}]
\draw (-2,0.3) node[anchor=center] {{$- \frac{1}{6}$}};
\node [lab] (v1) at (0,0) {\tiny $1$};
\node [lab] (v2) at (-1,1) {\tiny $2$};
\node [lab] (v3) at (0,1) {\tiny $3$};
\node [lab] (v4) at (1,1) {\tiny $4$};
\draw (v1) edge (v2) edge (v3) edge (v4) edge (0,-0.7);
\draw (4,0.3) node[anchor=center] {{\small $\otimes ~ b_1 \{b_3, \{ b_2, b_4 \} \}$}};
\end{scope}
\begin{scope}[shift={(0,-3.5)}]
\draw (-2,0.3) node[anchor=center] {{$- \frac{1}{12}$}};
\node [lab] (v1) at (0,0) {\tiny $1$};
\node [lab] (v2) at (-0.5,0.7) {\tiny $2$};
\node [lab] (v4) at (1,1.4) {\tiny $4$};
\node [lab] (v3) at (0.5,0.7) {\tiny $3$};
\draw (v1) edge (v2) edge (v3) edge (0,-0.7) (v3) edge (v4);
\draw (4,0.3) node[anchor=center] {{\small $\otimes ~ b_1 \{b_3, \{ b_2, b_4 \} \}$}};
\end{scope}
\begin{scope}[shift={(10,-3.5)}]
\draw (-2,0.3) node[anchor=center] {{$+ \frac{1}{24}$}};
\node [lab] (v1) at (0,0) {\tiny $1$};
\node [lab] (v2) at (-0.5,0.7) {\tiny $2$};
\node [lab] (v3) at (-1,1.4) {\tiny $3$};
\node [lab] (v4) at (0.5,0.7) {\tiny $4$};
\draw (v1) edge (v2) edge (v4) edge (0,-0.7) (v2) edge (v3);
\draw (4,0.3) node[anchor=center] {{\small $\otimes ~ \{ b_1, b_2\} \{ b_3, b_4 \}$}};
\end{scope}
\begin{scope}[shift={(20,-3.5)}]
\draw (-2,0.3) node[anchor=center] {{$+ \frac{1}{24}$}};
\node [lab] (v1) at (0,0) {\tiny $1$};
\node [lab] (v2) at (-0.5,0.7) {\tiny $2$};
\node [lab] (v3) at (-1,1.4) {\tiny $3$};
\node [lab] (v4) at (0.5,0.7) {\tiny $4$};
\draw (v1) edge (v2) edge (v4) edge (0,-0.7) (v2) edge (v3);
\draw (4,0.3) node[anchor=center] {{\small $\otimes ~ \{ b_1, b_3\} \{ b_2, b_4 \}$}};
\end{scope}
\begin{scope}[shift={(0,-7)}]
\draw (-2,0.3) node[anchor=center] {{$+ \frac{1}{24}$}};
\node [lab] (v1) at (0,0) {\tiny $1$};
\node [lab] (v2) at (-0.5,0.7) {\tiny $2$};
\node [lab] (v4) at (1,1.4) {\tiny $4$};
\node [lab] (v3) at (0.5,0.7) {\tiny $3$};
\draw (v1) edge (v2) edge (v3) edge (0,-0.7) (v3) edge (v4);
\draw (4,0.3) node[anchor=center] {{\small $\otimes ~ \{ b_1, b_3 \} \{ b_2, b_4 \} $}};
\end{scope}
\begin{scope}[shift={(10,-7)}]
\draw (-2,0.3) node[anchor=center] {{$+ \frac{1}{24}$}};
\node [lab] (v1) at (0,-0.3) {\tiny $1$};
\node [lab] (v2) at (0,0.4) {\tiny $2$};
\node [lab] (v3) at (0,1.1) {\tiny $3$};
\node [lab] (v4) at (0,1.8) {\tiny $4$};
\draw (v1) edge (v2) edge (0,-1) (v3) edge (v2) edge (v4);
\draw (4,0.3) node[anchor=center] {{\small $\otimes ~ \{ b_1, b_3 \} \{ b_2, b_4 \} $}};
\end{scope}
\begin{scope}[shift={(20,-7)}]
\draw (-2,0.3) node[anchor=center] {{$+ \frac{1}{24}$}};
\node [lab] (v1) at (0,0) {\tiny $1$};
\node [lab] (v2) at (-0.5,0.7) {\tiny $2$};
\node [lab] (v4) at (1,1.4) {\tiny $4$};
\node [lab] (v3) at (0.5,0.7) {\tiny $3$};
\draw (v1) edge (v2) edge (v3) edge (0,-0.7) (v3) edge (v4);
\draw (4,0.3) node[anchor=center] {{\small $\otimes ~ \{ b_1, b_4 \} \{ b_2, b_3 \} $}};
\end{scope}
\begin{scope}[shift={(0,-10.5)}]
\draw (-2,0.3) node[anchor=center] {{$+ \frac{1}{24}$}};
\node [lab] (v1) at (0,-0.3) {\tiny $1$};
\node [lab] (v2) at (0,0.4) {\tiny $2$};
\node [lab] (v3) at (0,1.1) {\tiny $3$};
\node [lab] (v4) at (0,1.8) {\tiny $4$};
\draw (v1) edge (v2) edge (0,-1) (v3) edge (v2) edge (v4);
\draw (4,0.3) node[anchor=center] {{\small $\otimes ~ \{ b_1, b_4 \} \{ b_2, b_3 \} $}};
\end{scope}
\begin{scope}[shift={(10,-10.5)}, color = blue] 
\draw (-2,0.3) node[anchor=center] {{$+ \frac{1}{36}$}};
\node [lab] (v1) at (0,0) {\tiny $1$};
\node [lab] (v2) at (-0.5,0.7) {\tiny $2$};
\node [lab] (v4) at (1,1.4) {\tiny $4$};
\node [lab] (v3) at (0.5,0.7) {\tiny $3$};
\draw (v1) edge (v2) edge (v3) edge (0,-0.7) (v3) edge (v4);
\draw (4,0.3) node[anchor=center] {{\small $\otimes ~ \{ b_1, b_2 \} \{ b_3, b_4 \} $}};
\end{scope}
\begin{scope}[shift={(20,-10.5)}, color = blue] 
\draw (-2,0.3) node[anchor=center] {{$+ \frac{1}{36}$}};
\node [lab] (v1) at (0,-0.3) {\tiny $1$};
\node [lab] (v2) at (0,0.4) {\tiny $2$};
\node [lab] (v3) at (0,1.1) {\tiny $3$};
\node [lab] (v4) at (0,1.8) {\tiny $4$};
\draw (v1) edge (v2) edge (0,-1) (v3) edge (v2) edge (v4);
\draw (4,0.3) node[anchor=center] {{\small $\otimes ~ \{ b_1, b_2 \} \{ b_3, b_4 \} $}};
\end{scope}
\begin{scope}[shift={(0,-14)}, color = blue] 
\draw (-2,0.3) node[anchor=center] {{$+ \frac{1}{36}$}};
\node [lab] (v1) at (0,0) {\tiny $1$};
\node [lab] (v2) at (-0.5,0.7) {\tiny $2$};
\node [lab] (v3) at (-1,1.4) {\tiny $3$};
\node [lab] (v4) at (0.5,0.7) {\tiny $4$};
\draw (v1) edge (v2) edge (v4) edge (0,-0.7) (v2) edge (v3);
\draw (4,0.3) node[anchor=center] {{\small $\otimes ~ \{ b_1, b_4 \} \{ b_2, b_3 \} $}};
\end{scope}
\begin{scope}[shift={(10,-14)}]
\draw (-2,0.3) node[anchor=center] {{$- \frac{1}{4}$}};
\node [n] (n) at (0,0) {};
\node [lab] (v1) at (-1.5,1.3) {\tiny $1$};
\node [lab] (v2) at (-0.6,1.3) {\tiny $2$};
\node [lab] (v3) at (0.6,1.3) {\tiny $3$};
\node [lab] (v4) at (1.5,1.3) {\tiny $4$};
\draw (n) edge (v1) edge (v2) edge (v3) edge (v4) edge (0,-0.7);
\draw (4.2,0.3) node[anchor=center] {{\small $\otimes ~ \{ b_1, \{ b_2, \{ b_3, b_4 \} \} \} $}};
\end{scope}
\begin{scope}[shift={(20,-14)}]
\draw (-2,0.3) node[anchor=center] {{$- \frac{13}{144}$}};
\node [n] (n) at (0,0) {};
\node [lab] (v1) at (-1,1) {\tiny $1$};
\node [lab] (v2) at (-1.6,1.7) {\tiny $2$};
\node [lab] (v3) at (0,1) {\tiny $3$};
\node [lab] (v4) at (1,1) {\tiny $4$};
\draw (n) edge (v1)  edge (v3) edge (v4) edge (0,-0.9) (v1) edge (v2);
\draw (4,0.3) node[anchor=center] {{\small $\otimes ~ \{ b_1, \{ b_2, \{ b_3, b_4 \} \} \} $}};
\end{scope}
\begin{scope}[shift={(0,-17.5)}]
\draw (-2,0.3) node[anchor=center] {{$- \frac{5}{48}$}};
\node [n] (n) at (0,0) {};
\node [lab] (v1) at (-1,1) {\tiny $1$};
\node [lab] (v2) at (0,1) {\tiny $2$};
\node [lab] (v3) at (0,1.9) {\tiny $3$};
\node [lab] (v4) at (1,1) {\tiny $4$};
\draw (n) edge (v1)  edge (v2) edge (v4) edge (0,-0.5) (v2) edge (v3);
\draw (4,0.3) node[anchor=center] {{\small $\otimes ~ \{ b_1, \{ b_2, \{ b_3, b_4 \} \} \} $}};
\end{scope}
\begin{scope}[shift={(10,-17.5)}]
\draw (-2.1,0.3) node[anchor=center] {{$- \frac{137}{1440}$}};
\node [n] (n) at (0,0) {};
\node [lab] (v1) at (-0.5,0.7) {\tiny $1$};
\node [lab] (v2) at (-1,1.4) {\tiny $2$};
\node [lab] (v3) at (0,1.4) {\tiny $3$};
\node [lab] (v4) at (0.5,0.7) {\tiny $4$};
\draw (n) edge (v1)  edge (v4) edge (0,-0.5) (v1) edge (v2) edge (v3);
\draw (4,0.3) node[anchor=center] {{\small $\otimes ~ \{ b_1, \{ b_2, \{ b_3, b_4 \} \} \} $}};
\end{scope}
\begin{scope}[shift={(20,-17.5)}]
\draw (-2.1,0.3) node[anchor=center] {{$- \frac{1}{36}$}};
\node [lab] (v1) at (0,0) {\tiny $1$};
\node [n] (n) at (-0.5,0.7) {};
\node [lab] (v4) at (0.5,0.7) {\tiny $4$};
\node [lab] (v2) at (-1,1.4) {\tiny $2$};
\node [lab] (v3) at (0,1.4) {\tiny $3$};
\draw (v1)  edge (n) edge (0,-0.7) edge (v4) (n) edge (v2) edge (v3);
\draw (4,0.3) node[anchor=center] {{\small $\otimes ~ \{ b_1, \{ b_2, \{ b_3, b_4 \} \} \} $}};
\end{scope}
\begin{scope}[shift={(0,-21)}]
\draw (-2.1,0.3) node[anchor=center] {{$- \frac{7}{360}$}};
\node [n] (n) at (0,0) {};
\node [lab] (v1) at (-0.5,0.5) {\tiny $1$};
\node [lab] (v2) at (-1,1) {\tiny $2$};
\node [lab] (v3) at (-1.5,1.5) {\tiny $3$};
\node [lab] (v4) at (0.5,0.5) {\tiny $4$};
\draw (n)  edge (v1) edge (0,-0.5) edge (v4) (v2) edge (v1) edge (v3);
\draw (4,0.3) node[anchor=center] {{\small $\otimes ~ \{ b_1, \{ b_2, \{ b_3, b_4 \} \} \} $}};
\end{scope}
\begin{scope}[shift={(10,-21)}]
\draw (-2.1,0.3) node[anchor=center] {{$- \frac{1}{36}$}};
\node [lab] (v1) at (0,-0.1) {\tiny $1$};
\node [n] (n) at (0,0.7) {};
\node [lab] (v2) at (-0.7,1.4) {\tiny $2$};
\node [lab] (v3) at (0,1.4) {\tiny $3$};
\node [lab] (v4) at (0.7,1.4) {\tiny $4$};
\draw (v1) edge (0,-0.7) (n) edge (v1) edge (v2) edge (v3) edge (v4);
\draw (4,0.3) node[anchor=center] {{\small $\otimes ~ \{ b_1, \{ b_2, \{ b_3, b_4 \} \} \} $}};
\end{scope}
\begin{scope}[shift={(20,-21)}]
\draw (-2.1,0.3) node[anchor=center] {{$- \frac{1}{20}$}};
\node [lab] (v1) at (0,-0.1) {\tiny $1$};
\node [n] (n) at (0,0.7) {};
\node [lab] (v2) at (-0.5,1.4) {\tiny $2$};
\node [lab] (v3) at (-1, 2.1) {\tiny $3$};
\node [lab] (v4) at (0.5,1.4) {\tiny $4$};
\draw (v1) edge (0,-0.7) (n) edge (v1) edge (v2) edge (v4) (v2) edge (v3);
\draw (4,0.3) node[anchor=center] {{\small $\otimes ~ \{ b_1, \{ b_2, \{ b_3, b_4 \} \} \} $}};
\end{scope}
\begin{scope}[shift={(0,-24)}]
\draw (-2,0.3) node[anchor=center] {{$- \frac{1}{9}$}};
\node [n] (n) at (0,0) {};
\node [lab] (v1) at (-1.5,1.3) {\tiny $1$};
\node [lab] (v2) at (-0.6,1.3) {\tiny $2$};
\node [lab] (v3) at (0.6,1.3) {\tiny $3$};
\node [lab] (v4) at (1.5,1.3) {\tiny $4$};
\draw (n) edge (v1) edge (v2) edge (v3) edge (v4) edge (0,-0.5);
\draw (4.2,0.3) node[anchor=center] {{\small $\otimes ~ \{ b_1, \{ b_3, \{ b_2, b_4 \} \} \} $}};
\end{scope}
\begin{scope}[shift={(10,-24)}]
\draw (-2,0.3) node[anchor=center] {{$- \frac{7}{60}$}};
\node [n] (n) at (0,0) {};
\node [lab] (v1) at (-0.7,0.7) {\tiny $1$};
\node [lab] (v3) at (0,0.7) {\tiny $3$};
\node [lab] (v4) at (0.7,0.7) {\tiny $4$};
\node [lab] (v2) at (-1.2,1.2) {\tiny $2$};
\draw (n) edge (v1) edge (v3) edge (v4) edge (0,-0.5) (v1) edge (v2);
\draw (4.2,0.3) node[anchor=center] {{\small $\otimes ~ \{ b_1, \{ b_3, \{ b_2, b_4 \} \} \} $}};
\end{scope}
\begin{scope}[shift={(20,-24)}]
\draw (-2,0.3) node[anchor=center] {{$- \frac{11}{720}$}};
\node [n] (n) at (0,0) {};
\node [lab] (v1) at (-0.7,0.7) {\tiny $1$};
\node [lab] (v2) at (0,0.7) {\tiny $2$};
\node [lab] (v4) at (0.7,0.7) {\tiny $4$};
\node [lab] (v3) at (0,1.4) {\tiny $3$};
\draw (n) edge (v1) edge (v2) edge (v4) edge (0,-0.5) (v2) edge (v3);
\draw (4.2,0.3) node[anchor=center] {{\small $\otimes ~ \{ b_1, \{ b_3, \{ b_2, b_4 \} \} \} $}};
\end{scope}
\begin{scope}[shift={(0,-27.5)}]
\draw (-2,0.3) node[anchor=center] {{$- \frac{157}{1440}$}};
\node [n] (n) at (0,0) {};
\node [lab] (v1) at (-0.5,0.7) {\tiny $1$};
\node [lab] (v2) at (-1,1.4) {\tiny $2$};
\node [lab] (v3) at (0,1.4) {\tiny $3$};
\node [lab] (v4) at (0.5,0.7) {\tiny $4$};
\draw (n) edge (v1)  edge (v4) edge (0,-0.5) (v1) edge (v2) edge (v3);
\draw (4.2,0.3) node[anchor=center] {{\small $\otimes ~ \{ b_1, \{ b_3, \{ b_2, b_4 \} \} \} $}};
\end{scope}
\begin{scope}[shift={(10,-27.5)}]
\draw (-2,0.3) node[anchor=center] {{$- \frac{1}{18}$}};
\node [lab] (v1) at (0,0) {\tiny $1$};
\node [n] (n) at (-0.5,0.7) {};
\node [lab] (v4) at (0.5,0.7) {\tiny $4$};
\node [lab] (v2) at (-1,1.4) {\tiny $2$};
\node [lab] (v3) at (0,1.4) {\tiny $3$};
\draw (v1)  edge (n) edge (0,-0.7) edge (v4) (n) edge (v2) edge (v3);
\draw (4.2,0.3) node[anchor=center] {{\small $\otimes ~ \{ b_1, \{ b_3, \{ b_2, b_4 \} \} \} $}};
\end{scope}
\begin{scope}[shift={(21,-27.5)}]
\draw (-2.5,0.3) node[anchor=center] {{$- \frac{11}{480}$}};
\node [n] (n) at (0,0) {};
\node [lab] (v1) at (-0.5,0.5) {\tiny $1$};
\node [lab] (v2) at (-1,1) {\tiny $2$};
\node [lab] (v3) at (-1.5,1.5) {\tiny $3$};
\node [lab] (v4) at (0.5,0.5) {\tiny $4$};
\draw (n)  edge (v1) edge (0,-0.5) edge (v4) (v2) edge (v1) edge (v3);
\draw (4.2,0.3) node[anchor=center] {{\small $\otimes ~ \{ b_1, \{ b_3, \{ b_2, b_4 \} \} \} $}};
\end{scope}
\begin{scope}[shift={(0,-31)}]
\draw (-2,0.3) node[anchor=center] {{$-\frac{1}{18}$}};
\node [n] (n) at (0,0) {};
\node [lab] (v1) at (-0.7,0.7) {\tiny $1$};
\node [lab] (v2) at (0,0.7) {\tiny $2$};
\node [lab] (v3) at (0.7,0.7) {\tiny $3$};
\node [lab] (v4) at (1.3,1.3) {\tiny $4$};
\draw (n)  edge (v1) edge (v2) edge (v3) edge (0,-0.5) (v3) edge (v4);
\draw (4.2,0.3) node[anchor=center] {{\small $\otimes ~ \{ b_1, \{ b_3, \{ b_2, b_4 \} \} \} $}};
\end{scope}
\begin{scope}[shift={(10,-31)}]
\draw (-2,0.3) node[anchor=center] {{$-\frac{1}{16}$}};
\node [n] (n) at (0,0) {};
\node [lab] (v1) at (-0.5,0.7) {\tiny $1$};
\node [lab] (v3) at (0.5,0.7) {\tiny $3$};
\node [lab] (v2) at (-1,1.4) {\tiny $2$};
\node [lab] (v4) at (1,1.4) {\tiny $4$};
\draw (n)  edge (v1) edge (v3) edge (0,-0.5) (v1) edge (v2) (v3) edge (v4);
\draw (4.2,0.3) node[anchor=center] {{\small $\otimes ~ \{ b_1, \{ b_3, \{ b_2, b_4 \} \} \} $}};
\end{scope}
\begin{scope}[shift={(20,-31)}]
\draw (-2,0.3) node[anchor=center] {{$-\frac{1}{45}$}};
\node [lab] (v1) at (0,0) {\tiny $1$};
\node [lab] (v2) at (-0.5,0.7) {\tiny $2$};
\node [n] (n) at (0.5,0.7) {};
\node [lab] (v3) at (0,1.4) {\tiny $3$};
\node [lab] (v4) at (1,1.4) {\tiny $4$};
\draw (v1) edge (v2) edge (n) edge (0,-0.5) (n) edge (v3) edge (v4);
\draw (4.2,0.3) node[anchor=center] {{\small $\otimes ~ \{ b_1, \{ b_3, \{ b_2, b_4 \} \} \} $}};
\end{scope}
\begin{scope}[shift={(0,-34.5)}]
\draw (-2,0.3) node[anchor=center] {{$-\frac{13}{720}$}};
\node [lab] (v1) at (0,-0.1) {\tiny $1$};
\node [n] (n) at (0,0.7) {};
\node [lab] (v2) at (-0.7,1.4) {\tiny $2$};
\node [lab] (v3) at (0,1.4) {\tiny $3$};
\node [lab] (v4) at (0.7,1.4) {\tiny $4$};
\draw (v1) edge (0,-0.6) (n) edge (v1) edge (v2) edge (v3) edge (v4);
\draw (4.2,0.3) node[anchor=center] {{\small $\otimes ~ \{ b_1, \{ b_3, \{ b_2, b_4 \} \} \} $}};
\end{scope}
\begin{scope}[shift={(10,-34.5)}]
\draw (-2,0.3) node[anchor=center] {{$-\frac{1}{24}$}};
\node [lab] (v1) at (0,-0.1) {\tiny $1$};
\node [n] (n) at (0,0.7) {};
\node [lab] (v2) at (-0.5,1.4) {\tiny $2$};
\node [lab] (v3) at (-1, 2.1) {\tiny $3$};
\node [lab] (v4) at (0.5,1.4) {\tiny $4$};
\draw (v1) edge (0,-0.7) (n) edge (v1) edge (v2) edge (v4) (v2) edge (v3);
\draw (4.2,0.3) node[anchor=center] {{\small $\otimes ~ \{ b_1, \{ b_3, \{ b_2, b_4 \} \} \} $}};
\end{scope}
\begin{scope}[shift={(20,-34.5)}]
\draw (-2,0.3) node[anchor=center] {{$-\frac{1}{144}$}};
\node [lab] (v1) at (0,-0.1) {\tiny $1$};
\node [n] (n) at (0,0.7) {};
\node [lab] (v2) at (-0.5,1.4) {\tiny $2$};
\node [lab] (v3) at (0.5, 1.4) {\tiny $3$};
\node [lab] (v4) at (1,2.1) {\tiny $4$};
\draw (v1) edge (0,-0.7) (n) edge (v1) edge (v2) edge (v3) (v3) edge (v4);
\draw (4.2,0.3) node[anchor=center] {{\small $\otimes ~ \{ b_1, \{ b_3, \{ b_2, b_4 \} \} \} $}};
\end{scope}
\end{tikzpicture}
\caption{The first part of $\al_4$} \label{fig:Part1}
\end{figure}
%
%
\begin{figure}[htp] 
\centering 
\begin{tikzpicture}[scale=0.5]
\tikzstyle{lab}=[circle, draw, minimum size=5, inner sep=1]
\tikzstyle{n}=[circle, draw, fill, minimum size=5, inner sep=1]
\tikzstyle{root}=[circle, draw, fill, minimum size=0, inner sep=1]
\begin{scope}[shift={(0,7)}]
\draw (-2,0.3) node[anchor=center] {{$+ \frac{7}{144}$}};
\node [n] (n) at (0,0) {};
\node [lab] (v1) at (-0.7,0.7) {\tiny $1$};
\node [lab] (v3) at (0,0.7) {\tiny $3$};
\node [lab] (v4) at (0.7,0.7) {\tiny $4$};
\node [lab] (v2) at (-1.2,1.2) {\tiny $2$};
\draw (n) edge (v1) edge (v3) edge (v4) edge (0,-0.5) (v1) edge (v2);
\draw (4.2,0.3) node[anchor=center] {{\small $\otimes ~ \{ b_2, \{ b_1, \{ b_3, b_4 \} \} \} $}};
\end{scope}
\begin{scope}[shift={(10,7)}]
\draw (-2,0.3) node[anchor=center] {{$+ \frac{1}{45}$}};
\node [n] (n) at (0,0) {};
\node [lab] (v1) at (-0.7,0.7) {\tiny $1$};
\node [lab] (v2) at (0,0.7) {\tiny $2$};
\node [lab] (v4) at (0.7,0.7) {\tiny $4$};
\node [lab] (v3) at (0,1.4) {\tiny $3$};
\draw (n) edge (v1) edge (v2) edge (v4) edge (0,-0.5) (v2) edge (v3);
\draw (4.2,0.3) node[anchor=center] {{\small $\otimes ~ \{ b_2, \{ b_1, \{ b_3, b_4 \} \} \} $}};
\end{scope}
\begin{scope}[shift={(20,7)}]
\draw (-2,0.3) node[anchor=center] {{$+ \frac{1}{6}$}};
\node [n] (n) at (0,0) {};
\node [lab] (v1) at (-1.5,1.3) {\tiny $1$};
\node [lab] (v2) at (-0.5,1.3) {\tiny $2$};
\node [lab] (v3) at (0.5,1.3) {\tiny $3$};
\node [lab] (v4) at (1.5,1.3) {\tiny $4$};
\draw (n) edge (v1) edge (v2) edge (v3) edge (v4) edge (0,-0.5);
\draw (4.2,0.3) node[anchor=center] {{\small $\otimes ~ \{ b_2, \{ b_3, \{ b_1, b_4 \} \} \} $}};
\end{scope}
\begin{scope}[shift={(0,3.5)}]
\draw (-2,0.3) node[anchor=center] {{$+ \frac{1}{45}$}};
\node [n] (n) at (0,0) {};
\node [lab] (v1) at (-0.7,0.7) {\tiny $1$};
\node [lab] (v3) at (0,0.7) {\tiny $3$};
\node [lab] (v4) at (0.7,0.7) {\tiny $4$};
\node [lab] (v2) at (-1.2,1.2) {\tiny $2$};
\draw (n) edge (v1) edge (v3) edge (v4) edge (0,-0.5) (v1) edge (v2);
\draw (4.2,0.3) node[anchor=center] {{\small $\otimes ~ \{ b_2, \{ b_3, \{ b_1, b_4 \} \} \} $}};
\end{scope}
\begin{scope}[shift={(10,3.5)}]
\draw (-2,0.3) node[anchor=center] {{$+ \frac{1}{8}$}};
\node [n] (n) at (0,0) {};
\node [lab] (v1) at (-0.7,0.7) {\tiny $1$};
\node [lab] (v2) at (0,0.7) {\tiny $2$};
\node [lab] (v4) at (0.7,0.7) {\tiny $4$};
\node [lab] (v3) at (0,1.4) {\tiny $3$};
\draw (n) edge (v1) edge (v2) edge (v4) edge (0,-0.5) (v2) edge (v3);
\draw (4.2,0.3) node[anchor=center] {{\small $\otimes ~ \{ b_2, \{ b_3, \{ b_1, b_4 \} \} \} $}};
\end{scope}
\begin{scope}[shift={(20,3.5)}]
\draw (-2,0.3) node[anchor=center] {{$+ \frac{1}{12}$}};
\node [n] (n) at (0,0) {};
\node [lab] (v1) at (-0.7,0.7) {\tiny $1$};
\node [lab] (v2) at (0,0.7) {\tiny $2$};
\node [lab] (v3) at (0.7,0.7) {\tiny $3$};
\node [lab] (v4) at (1.3,1.3) {\tiny $4$};
\draw (n)  edge (v1) edge (v2) edge (v3) edge (0,-0.5) (v3) edge (v4);
\draw (4.2,0.3) node[anchor=center] {{\small $\otimes ~ \{ b_2, \{ b_3, \{ b_1, b_4 \} \} \} $}};
\end{scope}
\begin{scope}[shift={(0,0)}]
\draw (-2,0.3) node[anchor=center] {{$+ \frac{1}{48}$}};
\node [n] (n) at (0,0) {};
\node [lab] (v1) at (-0.5,0.7) {\tiny $1$};
\node [lab] (v3) at (0.5,0.7) {\tiny $3$};
\node [lab] (v2) at (-1,1.4) {\tiny $2$};
\node [lab] (v4) at (1,1.4) {\tiny $4$};
\draw (n)  edge (v1) edge (v3) edge (0,-0.5) (v1) edge (v2) (v3) edge (v4);
\draw (4.2,0.3) node[anchor=center] {{\small $\otimes ~ \{ b_2, \{ b_3, \{ b_1, b_4 \} \} \} $}};
\end{scope}
\begin{scope}[shift={(10,0)}]
\draw (-2,0.3) node[anchor=center] {{$+ \frac{5}{72}$}};
\node [n] (n) at (0,0) {};
\node [lab] (v1) at (-0.5,0.7) {\tiny $1$};
\node [lab] (v2) at (0.5,0.7) {\tiny $2$};
\node [lab] (v3) at (0,1.4) {\tiny $3$};
\node [lab] (v4) at (1,1.4) {\tiny $4$};
\draw (n)  edge (v1) edge (v2) edge (0,-0.5) (v2) edge (v3) edge (v4);
\draw (4.2,0.3) node[anchor=center] {{\small $\otimes ~ \{ b_2, \{ b_3, \{ b_1, b_4 \} \} \} $}};
\end{scope}
\begin{scope}[shift={(20,0)}]
\draw (-2,0.3) node[anchor=center] {{$+ \frac{1}{24}$}};
\node [n] (n) at (0,0) {};
\node [lab] (v1) at (-0.4,0.6) {\tiny $1$};
\node [lab] (v2) at (0.4,0.6) {\tiny $2$};
\node [lab] (v3) at (0.8,1.2) {\tiny $3$};
\node [lab] (v4) at (1.2,1.8) {\tiny $4$};
\draw (n)  edge (v1) edge (v2) edge (0,-0.5) (v3) edge (v2) edge (v4);
\draw (4.2,0.3) node[anchor=center] {{\small $\otimes ~ \{ b_2, \{ b_3, \{ b_1, b_4 \} \} \} $}};
\end{scope}
\begin{scope}[shift={(0,-3.5)}]
\draw (-2,0.3) node[anchor=center] {{$+ \frac{1}{288}$}};
\node [n] (n) at (0,0) {};
\node [lab] (v1) at (-0.5,0.7) {\tiny $1$};
\node [lab] (v2) at (0.5,0.7) {\tiny $2$};
\node [lab] (v3) at (0,1.4) {\tiny $3$};
\node [lab] (v4) at (1,1.4) {\tiny $4$};
\draw (n)  edge (v1) edge (v2) edge (0,-0.5) (v2) edge (v3) edge (v4);
\draw (4.2,0.3) node[anchor=center] {{\small $\otimes ~ \{ b_1, \{ b_2, \{ b_3, b_4 \} \} \} $}};
\end{scope}
\begin{scope}[shift={(10,-3.5)}]
\draw (-2,0.3) node[anchor=center] {{$- \frac{1}{720}$}};
\node [n] (n) at (0,0) {};
\node [lab] (v1) at (-0.4,0.6) {\tiny $1$};
\node [lab] (v2) at (0.4,0.6) {\tiny $2$};
\node [lab] (v3) at (0.8,1.2) {\tiny $3$};
\node [lab] (v4) at (1.2,1.8) {\tiny $4$};
\draw (n)  edge (v1) edge (v2) edge (0,-0.5) (v3) edge (v2) edge (v4);
\draw (4.2,0.3) node[anchor=center] {{\small $\otimes ~ \{ b_1, \{ b_2, \{ b_3, b_4 \} \} \} $}};
\end{scope}
\begin{scope}[shift={(20,-3.5)}]
\draw (-2,0.3) node[anchor=center] {{$ + \frac{7}{180}$}};
\node [n] (n) at (0,0) {};
\node [lab] (v1) at (-0.5,0.7) {\tiny $1$};
\node [lab] (v2) at (0.5,0.7) {\tiny $2$};
\node [lab] (v3) at (0,1.4) {\tiny $3$};
\node [lab] (v4) at (1,1.4) {\tiny $4$};
\draw (n)  edge (v1) edge (v2) edge (0,-0.5) (v2) edge (v3) edge (v4);
\draw (4.2,0.3) node[anchor=center] {{\small $\otimes ~ \{ b_1, \{ b_3, \{ b_2, b_4 \} \} \} $}};
\end{scope}
\begin{scope}[shift={(0,-7)}]
\draw (-2,0.3) node[anchor=center] {{$- \frac{1}{288}$}};
\node [n] (n) at (0,0) {};
\node [lab] (v1) at (-0.4,0.6) {\tiny $1$};
\node [lab] (v2) at (0.4,0.6) {\tiny $2$};
\node [lab] (v3) at (0.8,1.2) {\tiny $3$};
\node [lab] (v4) at (1.2,1.8) {\tiny $4$};
\draw (n)  edge (v1) edge (v2) edge (0,-0.5) (v3) edge (v2) edge (v4);
\draw (4.2,0.3) node[anchor=center] {{\small $\otimes ~ \{ b_1, \{ b_3, \{ b_2, b_4 \} \} \} $}};
\end{scope}
\begin{scope}[shift={(10,-7)}]
\draw (-2,0.3) node[anchor=center] {{$ + \frac{1}{72}$}};
\node [n] (n) at (0,0) {};
\node [lab] (v1) at (-0.5,0.7) {\tiny $1$};
\node [lab] (v3) at (0.5,0.7) {\tiny $3$};
\node [lab] (v2) at (-1,1.4) {\tiny $2$};
\node [lab] (v4) at (1,1.4) {\tiny $4$};
\draw (n)  edge (v1) edge (v3) edge (0,-0.5) (v1) edge (v2) (v3) edge (v4);
\draw (4.2,0.3) node[anchor=center] {{\small $\otimes ~ \{ b_1, \{ b_2, \{ b_3, b_4 \} \} \} $}};
\end{scope}
\begin{scope}[shift={(20,-7)}]
\draw (-2,0.3) node[anchor=center] {{$ - \frac{1}{72}$}};
\node [n] (n) at (0,0) {};
\node [lab] (v1) at (-0.5,0.7) {\tiny $1$};
\node [lab] (v3) at (0.5,0.7) {\tiny $3$};
\node [lab] (v2) at (-1,1.4) {\tiny $2$};
\node [lab] (v4) at (1,1.4) {\tiny $4$};
\draw (n)  edge (v1) edge (v3) edge (0,-0.5) (v1) edge (v2) (v3) edge (v4);
\draw (4.2,0.3) node[anchor=center] {{\small $\otimes ~ \{ b_2, \{ b_1, \{ b_3, b_4 \} \} \} $}};
\end{scope}
\begin{scope}[shift={(0,-10.5)}]
\draw (-2,0.3) node[anchor=center] {{$ + \frac{1}{36}$}};
\node [n] (n) at (0,0) {};
\node [lab] (v1) at (-1.3,1.3) {\tiny $1$};
\node [lab] (v2) at (-0.4,1.3) {\tiny $2$};
\node [lab] (v3) at (0.4,1.3) {\tiny $3$};
\node [lab] (v4) at (1.3,1.3) {\tiny $4$};
\draw (n) edge (v1) edge (v2) edge (v3) edge (v4) edge (0,-0.5);
\draw (4.2,0.3) node[anchor=center] {{\small $\otimes ~ \{ b_2, \{ b_1, \{ b_3, b_4 \} \} \} $}};
\end{scope}
\begin{scope}[shift={(10,-10.5)}]
\draw (-2,0.3) node[anchor=center] {{$ + \frac{1}{120}$}};
\node [n] (n) at (0,0) {};
\node [lab] (v1) at (-0.5,0.7) {\tiny $1$};
\node [lab] (v2) at (-1,1.4) {\tiny $2$};
\node [lab] (v3) at (0,1.4) {\tiny $3$};
\node [lab] (v4) at (0.5,0.7) {\tiny $4$};
\draw (n) edge (v1)  edge (v4) edge (0,-0.5) (v1) edge (v2) edge (v3);
\draw (4.2,0.3) node[anchor=center] {{\small $\otimes ~ \{ b_2, \{ b_1, \{ b_3, b_4 \} \} \} $}};
\end{scope}
\begin{scope}[shift={(20,-10.5)}]
\draw (-2,0.3) node[anchor=center] {{$ + \frac{1}{36}$}};
\node [lab] (v1) at (0,0) {\tiny $1$};
\node [n] (n) at (-0.5,0.7) {};
\node [lab] (v4) at (0.5,0.7) {\tiny $4$};
\node [lab] (v2) at (-1,1.4) {\tiny $2$};
\node [lab] (v3) at (0,1.4) {\tiny $3$};
\draw (v1)  edge (n) edge (0,-0.7) edge (v4) (n) edge (v2) edge (v3);
\draw (4.2,0.3) node[anchor=center] {{\small $\otimes ~ \{ b_2, \{ b_1, \{ b_3, b_4 \} \} \} $}};
\end{scope}
\begin{scope}[shift={(0,-14)}]
\draw (-2,0.3) node[anchor=center] {{$ - \frac{1}{72}$}};
\node [lab] (v1) at (0,0) {\tiny $1$};
\node [lab] (v2) at (-0.5,0.7) {\tiny $2$};
\node [n] (n) at (0.5,0.7) {};
\node [lab] (v3) at (0.1,1.4) {\tiny $3$};
\node [lab] (v4) at (0.9,1.4) {\tiny $4$};
\draw (v1) edge (v2) edge (n) edge (0,-0.5) (n) edge (v3) edge (v4);
\draw (4.2,0.3) node[anchor=center] {{\small $\otimes ~ \{ b_2, \{ b_1, \{ b_3, b_4 \} \} \} $}};
\end{scope}
\begin{scope}[shift={(10,-14)}]
\draw (-1.8,0.3) node[anchor=center] {{$ + \frac{1}{72}$}};
\node [lab] (v1) at (0,0) {\tiny $1$};
\node [n] (n) at (0,0.7) {};
\node [lab] (v2) at (-0.7,1.4) {\tiny $2$};
\node [lab] (v3) at (0,1.4) {\tiny $3$};
\node [lab] (v4) at (0.7,1.4) {\tiny $4$};
\draw (v1) edge (0,-0.5) (n) edge (v1) edge (v2) edge (v3) edge (v4);
\draw (4.2,0.3) node[anchor=center] {{\small $\otimes ~ \{ b_2, \{ b_1, \{ b_3, b_4 \} \} \} $}};
\end{scope}
\begin{scope}[shift={(20,-14)}]
\draw (-1.8,0.3) node[anchor=center] {{$ - \frac{19}{720}$}};
\node [lab] (v1) at (0,-0.15) {\tiny $1$};
\node [lab] (v2) at (0,0.6) {\tiny $2$};
\node [n] (n) at (0,1.25) {};
\node [lab] (v3) at (-0.4,1.8) {\tiny $3$};
\node [lab] (v4) at (0.4,1.8) {\tiny $4$};
\draw (v1) edge (0,-0.7) edge (v2) (n) edge (v2) edge (v3) edge (v4);
\draw (4.2,0.3) node[anchor=center] {{\small $\otimes ~ \{ b_1, \{ b_3, \{ b_2, b_4 \} \} \} $}};
\end{scope}
\begin{scope}[shift={(-0.2,-17.5)}]
\draw (-2.3,0.3) node[anchor=center] {{$ +\frac{1}{288}$}};
\node [n] (n) at (0,0) {};
\node [lab] (v1) at (-0.5,0.5) {\tiny $1$};
\node [lab] (v2) at (-1,1) {\tiny $2$};
\node [lab] (v3) at (-1.5,1.5) {\tiny $3$};
\node [lab] (v4) at (0.5,0.5) {\tiny $4$};
\draw (n)  edge (v1) edge (0,-0.5) edge (v4) (v2) edge (v1) edge (v3);
\draw (4.2,0.3) node[anchor=center] {{\small $\otimes ~ \{ b_2, \{ b_1, \{ b_3, b_4 \} \} \} $}};
\end{scope}
\begin{scope}[shift={(10,-17.5)}]
\draw (-2,0.3) node[anchor=center] {{$ -\frac{11}{1440}$}};
\node [n] (n) at (0,0) {};
\node [lab] (v1) at (-0.5,0.7) {\tiny $1$};
\node [lab] (v2) at (-1,1.4) {\tiny $2$};
\node [lab] (v3) at (0,1.4) {\tiny $3$};
\node [lab] (v4) at (0.5,0.7) {\tiny $4$};
\draw (n) edge (v1)  edge (v4) edge (0,-0.5) (v1) edge (v2) edge (v3);
\draw (4.2,0.3) node[anchor=center] {{\small $\otimes ~ \{ b_2, \{ b_3, \{ b_1, b_4 \} \} \} $}};
\end{scope}
\begin{scope}[shift={(20.2,-17.5)}]
\draw (-2,0.3) node[anchor=center] {{$ -\frac{1}{180}$}};
\node [n] (n) at (0,0) {};
\node [lab] (v1) at (-0.7,0.7) {\tiny $1$};
\node [lab] (v2) at (0,0.7) {\tiny $2$};
\node [lab] (v3) at (0.7,0.7) {\tiny $3$};
\node [lab] (v4) at (1.3,1.3) {\tiny $4$};
\draw (n)  edge (v1) edge (v2) edge (v3) edge (0,-0.5) (v3) edge (v4);
\draw (4.2,0.3) node[anchor=center] {{\small $\otimes ~ \{ b_1, \{ b_2, \{ b_3, b_4 \} \} \} $}};
\end{scope}
\begin{scope}[shift={(0,-21)}]
\draw (-2,0.3) node[anchor=center] {{$ -\frac{1}{30}$}};
\node [lab] (v1) at (0,0) {\tiny $1$};
\node [n] (n) at (-0.5,0.7) {};
\node [lab] (v4) at (0.5,0.7) {\tiny $4$};
\node [lab] (v2) at (-1,1.4) {\tiny $2$};
\node [lab] (v3) at (0,1.4) {\tiny $3$};
\draw (v1)  edge (n) edge (0,-0.7) edge (v4) (n) edge (v2) edge (v3);
\draw (4.2,0.3) node[anchor=center] {{\small $\otimes ~ \{ b_2, \{ b_3, \{ b_1, b_4 \} \} \} $}};
\end{scope}
\begin{scope}[shift={(10,-21)}]
\draw (-2,0.3) node[anchor=center] {{$ + \frac{1}{30}$}};
\node [lab] (v1) at (0,0) {\tiny $1$};
\node [lab] (v2) at (-0.5,0.7) {\tiny $2$};
\node [n] (n) at (0.5,0.7) {};
\node [lab] (v3) at (0.1,1.4) {\tiny $3$};
\node [lab] (v4) at (0.9,1.4) {\tiny $4$};
\draw (v1) edge (v2) edge (n) edge (0,-0.5) (n) edge (v3) edge (v4);
\draw (4.2,0.3) node[anchor=center] {{\small $\otimes ~ \{ b_2, \{ b_3, \{ b_1, b_4 \} \} \} $}};
\end{scope}
\begin{scope}[shift={(20,-21)}]
\draw (-2,0.3) node[anchor=center] {{$ - \frac{1}{36}$}};
\node [lab] (v1) at (0,0) {\tiny $1$};
\node [n] (n) at (0,0.7) {};
\node [lab] (v2) at (-0.7,1.4) {\tiny $2$};
\node [lab] (v3) at (0,1.4) {\tiny $3$};
\node [lab] (v4) at (0.7,1.4) {\tiny $4$};
\draw (v1) edge (0,-0.5) (n) edge (v1) edge (v2) edge (v3) edge (v4);
\draw (4.2,0.3) node[anchor=center] {{\small $\otimes ~ \{ b_2, \{ b_3, \{ b_1, b_4 \} \} \} $}};
\end{scope}
\begin{scope}[shift={(0,-24.5)}]
\draw (-2,0.3) node[anchor=center] {{$ -\frac{7}{360}$}};
\node [lab] (v1) at (0,-0.1) {\tiny $1$};
\node [n] (n) at (0,0.7) {};
\node [lab] (v2) at (-0.5,1.4) {\tiny $2$};
\node [lab] (v3) at (-1, 2.1) {\tiny $3$};
\node [lab] (v4) at (0.5,1.4) {\tiny $4$};
\draw (v1) edge (0,-0.7) (n) edge (v1) edge (v2) edge (v4) (v2) edge (v3);
\draw (4.2,0.3) node[anchor=center] {{\small $\otimes ~ \{ b_2, \{ b_1, \{ b_3, b_4 \} \} \} $}};
\end{scope}
\begin{scope}[shift={(10,-24.5)}]
\draw (-1.8,0.3) node[anchor=center] {{$+ \frac{1}{180}$}};
\node [lab] (v1) at (0,-0.15) {\tiny $1$};
\node [lab] (v2) at (0,0.6) {\tiny $2$};
\node [n] (n) at (0,1.25) {};
\node [lab] (v3) at (-0.4,1.8) {\tiny $3$};
\node [lab] (v4) at (0.4,1.8) {\tiny $4$};
\draw (v1) edge (0,-0.7) edge (v2) (n) edge (v2) edge (v3) edge (v4);
\draw (4.2,0.3) node[anchor=center] {{\small $\otimes ~ \{ b_2, \{ b_3, \{ b_1, b_4 \} \} \} $}};
\end{scope}
\begin{scope}[shift={(20,-24.5)}]
\draw (-2,0.3) node[anchor=center] {{$- \frac{1}{72}$}};
\node [lab] (v1) at (0,0) {\tiny $1$};
\node [lab] (v2) at (-0.5,0.7) {\tiny $2$};
\node [n] (n) at (0.5,0.7) {};
\node [lab] (v3) at (0.1,1.4) {\tiny $3$};
\node [lab] (v4) at (0.9,1.4) {\tiny $4$};
\draw (v1) edge (v2) edge (n) edge (0,-0.5) (n) edge (v3) edge (v4);
\draw (4.2,0.3) node[anchor=center] {{\small $\otimes ~ \{ b_1, \{ b_2, \{ b_3, b_4 \} \} \} $}};
\end{scope}
\begin{scope}[shift={(0,-28)}]
\draw (-2,0.3) node[anchor=center] {{$ +\frac{1}{72}$}};
\node [n] (n) at (0,0) {};
\node [lab] (v1) at (-0.7,0.7) {\tiny $1$};
\node [lab] (v2) at (0,0.7) {\tiny $2$};
\node [lab] (v4) at (0.7,0.7) {\tiny $4$};
\node [lab] (v3) at (0,1.4) {\tiny $3$};
\draw (n) edge (v1) edge (v2) edge (v4) edge (0,-0.5) (v2) edge (v3);
\draw (4.2,0.3) node[anchor=center] {{\small $\otimes ~ \{ b_3, \{ b_1, \{ b_2, b_4 \} \} \} $}};
\end{scope}
\begin{scope}[shift={(10,-28)}]
\draw (-2,0.3) node[anchor=center] {{$ +\frac{5}{144}$}};
\node [lab] (v1) at (0,-0.1) {\tiny $1$};
\node [n] (n) at (0,0.7) {};
\node [lab] (v2) at (-0.5,1.4) {\tiny $2$};
\node [lab] (v3) at (0.5, 1.4) {\tiny $3$};
\node [lab] (v4) at (1,2.1) {\tiny $4$};
\draw (v1) edge (0,-0.6) (n) edge (v1) edge (v2) edge (v3) (v3) edge (v4);
\draw (4.2,0.3) node[anchor=center] {{\small $\otimes ~ \{ b_1, \{ b_2, \{ b_3, b_4 \} \} \} $}};
\end{scope}
\begin{scope}[shift={(20,-28)}]
\draw (-2,0.3) node[anchor=center] {{$ +\frac{7}{360}$}};
\node [n] (n) at (0,0) {};
\node [lab] (v1) at (-0.5,0.7) {\tiny $1$};
\node [lab] (v2) at (-1,1.4) {\tiny $2$};
\node [lab] (v3) at (0,1.4) {\tiny $3$};
\node [lab] (v4) at (0.5,0.7) {\tiny $4$};
\draw (n) edge (v1)  edge (v4) edge (0,-0.5) (v1) edge (v2) edge (v3);
\draw (4.2,0.3) node[anchor=center] {{\small $\otimes ~ \{ b_3, \{ b_1, \{ b_2, b_4 \} \} \} $}};
\end{scope}
\begin{scope}[shift={(0,-31.5)}]
\draw (-2,0.3) node[anchor=center] {{$- \frac{5}{144}$}};
\node [lab] (v1) at (0,-0.15) {\tiny $1$};
\node [lab] (v2) at (0,0.6) {\tiny $2$};
\node [n] (n) at (0,1.25) {};
\node [lab] (v3) at (-0.4,1.8) {\tiny $3$};
\node [lab] (v4) at (0.4,1.8) {\tiny $4$};
\draw (v1) edge (0,-0.7) edge (v2) (n) edge (v2) edge (v3) edge (v4);
\draw (4,0.3) node[anchor=center] {{\small $\otimes ~ \{ b_1, \{ b_2, \{ b_3, b_4 \} \} \} $}};
\end{scope}
\begin{scope}[shift={(10,-31.5)}]
\draw (-2,0.3) node[anchor=center] {{$+ \frac{23}{1440}$}};
\node [n] (n) at (0,0) {};
\node [lab] (v1) at (-0.5,0.5) {\tiny $1$};
\node [lab] (v2) at (-1,1) {\tiny $2$};
\node [lab] (v3) at (-1.5,1.5) {\tiny $3$};
\node [lab] (v4) at (0.5,0.5) {\tiny $4$};
\draw (n)  edge (v1) edge (0,-0.5) edge (v4) (v2) edge (v1) edge (v3);
\draw (4.2,0.3) node[anchor=center] {{\small $\otimes ~ \{ b_2, \{ b_3, \{ b_1, b_4 \} \} \} $}};
\end{scope}
\begin{scope}[shift={(20,-31.5)}]
\draw (-2,0.3) node[anchor=center] {{$- \frac{1}{288}$}};
\node [n] (n) at (0,0) {};
\node [lab] (v1) at (-0.5,0.5) {\tiny $1$};
\node [lab] (v2) at (-1,1) {\tiny $2$};
\node [lab] (v3) at (-1.5,1.5) {\tiny $3$};
\node [lab] (v4) at (0.5,0.5) {\tiny $4$};
\draw (n)  edge (v1) edge (0,-0.5) edge (v4) (v2) edge (v1) edge (v3);
\draw (4.2,0.3) node[anchor=center] {{\small $\otimes ~ \{ b_3, \{ b_1, \{ b_2, b_4 \} \} \} $}};
\end{scope}
\begin{scope}[shift={(0,-34)}]
\draw (-2,0.3) node[anchor=center] {{$- \frac{1}{36}$}};
\node [n] (n) at (0,0) {};
\node [lab] (v1) at (-0.7,0.7) {\tiny $1$};
\node [lab] (v2) at (0,0.7) {\tiny $2$};
\node [lab] (v3) at (0.7,0.7) {\tiny $3$};
\node [lab] (v4) at (1.3,1.3) {\tiny $4$};
\draw (n)  edge (v1) edge (v2) edge (v3) edge (0,-0.5) (v3) edge (v4);
\draw (4.2,0.3) node[anchor=center] {{\small $\otimes ~ \{ b_2, \{ b_1, \{ b_3, b_4 \} \} \} $}};
\end{scope}
\begin{scope}[shift={(10,-34)}]
\draw (-2,0.3) node[anchor=center] {{$+ \frac{1}{72}$}};
\node [n] (n) at (0,0) {};
\node [lab] (v1) at (-1.3,1.3) {\tiny $1$};
\node [lab] (v2) at (-0.4,1.3) {\tiny $2$};
\node [lab] (v3) at (0.4,1.3) {\tiny $3$};
\node [lab] (v4) at (1.3,1.3) {\tiny $4$};
\draw (n) edge (v1) edge (v2) edge (v3) edge (v4) edge (0,-0.5);
\draw (4.2,0.3) node[anchor=center] {{\small $\otimes ~ \{ b_3, \{ b_1, \{ b_2, b_4 \} \} \} $}};
\end{scope}
\begin{scope}[shift={(20,-34)}]
\draw (-2,0.3) node[anchor=center] {{$+ \frac{1}{36}$}};
\node [n] (n) at (0,0) {};
\node [lab] (v1) at (-0.7,0.7) {\tiny $1$};
\node [lab] (v2) at (0,0.7) {\tiny $2$};
\node [lab] (v3) at (0.7,0.7) {\tiny $3$};
\node [lab] (v4) at (1.3,1.3) {\tiny $4$};
\draw (n)  edge (v1) edge (v2) edge (v3) edge (0,-0.5) (v3) edge (v4);
\draw (4.2,0.3) node[anchor=center] {{\small $\otimes ~ \{ b_3, \{ b_1, \{ b_2, b_4 \} \} \} $}};
\end{scope}
\end{tikzpicture}
\caption{The second part of $\al_4$} \label{fig:Part11}
\end{figure}
%
%
\begin{figure}[htp] 
\centering 
\begin{tikzpicture}[scale=0.5]
\tikzstyle{lab}=[circle, draw, minimum size=5, inner sep=1]
\tikzstyle{n}=[circle, draw, fill, minimum size=5, inner sep=1]
\tikzstyle{root}=[circle, draw, fill, minimum size=0, inner sep=1]
\begin{scope}[shift={(0,7)}]
\draw (-2,0.3) node[anchor=center] {{$+\frac{1}{72}$}};
\node [n] (n) at (0,0) {};
\node [lab] (v1) at (-0.5,0.7) {\tiny $1$};
\node [lab] (v2) at (0.5,0.7) {\tiny $2$};
\node [lab] (v3) at (0,1.4) {\tiny $3$};
\node [lab] (v4) at (1,1.4) {\tiny $4$};
\draw (n)  edge (v1) edge (v2) edge (0,-0.5) (v2) edge (v3) edge (v4);
\draw (4.2,0.3) node[anchor=center] {{\small $\otimes ~ \{ b_3, \{ b_1, \{ b_2, b_4 \} \} \} $}};
\end{scope}
\begin{scope}[shift={(10,7)}]
\draw (-2,0.3) node[anchor=center] {{$-\frac{1}{72}$}};
\node [lab] (v1) at (0,0) {\tiny $1$};
\node [lab] (v2) at (-0.5,0.7) {\tiny $2$};
\node [n] (n) at (0.5,0.7) {};
\node [lab] (v3) at (0.1,1.4) {\tiny $3$};
\node [lab] (v4) at (0.9,1.4) {\tiny $4$};
\draw (v1) edge (v2) edge (n) edge (0,-0.5) (n) edge (v3) edge (v4);
\draw (4.2,0.3) node[anchor=center] {{\small $\otimes ~ \{ b_3, \{ b_1, \{ b_2, b_4 \} \} \} $}};
\end{scope}
\begin{scope}[shift={(20,7)}]
\draw (-2,0.3) node[anchor=center] {{$+\frac{1}{72}$}};
\node [lab] (v1) at (0,0) {\tiny $1$};
\node [n] (n) at (0,0.7) {};
\node [lab] (v2) at (-0.7,1.4) {\tiny $2$};
\node [lab] (v3) at (0,1.4) {\tiny $3$};
\node [lab] (v4) at (0.7,1.4) {\tiny $4$};
\draw (v1) edge (0,-0.5) (n) edge (v1) edge (v2) edge (v3) edge (v4);
\draw (4.2,0.3) node[anchor=center] {{\small $\otimes ~ \{ b_3, \{ b_1, \{ b_2, b_4 \} \} \} $}};
\end{scope}
\begin{scope}[shift={(0,3.5)}]
\draw (-2,0.3) node[anchor=center] {{$+\frac{1}{72}$}};
\node [n] (n) at (0,0) {};
\node [lab] (v1) at (-0.4,0.6) {\tiny $1$};
\node [lab] (v2) at (0.4,0.6) {\tiny $2$};
\node [lab] (v3) at (0.8,1.2) {\tiny $3$};
\node [lab] (v4) at (1.2,1.8) {\tiny $4$};
\draw (n)  edge (v1) edge (v2) edge (0,-0.5) (v3) edge (v2) edge (v4);
\draw (4.2,0.3) node[anchor=center] {{\small $\otimes ~ \{ b_3, \{ b_1, \{ b_2, b_4 \} \} \} $}};
\end{scope}
\begin{scope}[shift={(10,3.5)}]
\draw (-2,0.3) node[anchor=center] {{$-\frac{1}{72}$}};
\node [n] (n) at (0,0) {};
\node [lab] (v1) at (-0.5,0.7) {\tiny $1$};
\node [lab] (v2) at (0.5,0.7) {\tiny $2$};
\node [lab] (v3) at (0,1.4) {\tiny $3$};
\node [lab] (v4) at (1,1.4) {\tiny $4$};
\draw (n)  edge (v1) edge (v2) edge (0,-0.5) (v2) edge (v3) edge (v4);
\draw (4.2,0.3) node[anchor=center] {{\small $\otimes ~ \{ b_2, \{ b_1, \{ b_3, b_4 \} \} \} $}};
\end{scope}
\begin{scope}[shift={(20,3.5)}]
\draw (-2,0.3) node[anchor=center] {{$-\frac{1}{72}$}};
\node [n] (n) at (0,0) {};
\node [lab] (v1) at (-1.3,1.3) {\tiny $1$};
\node [lab] (v2) at (-0.4,1.3) {\tiny $2$};
\node [lab] (v3) at (0.4,1.3) {\tiny $3$};
\node [lab] (v4) at (1.3,1.3) {\tiny $4$};
\draw (n) edge (v1) edge (v2) edge (v3) edge (v4) edge (0,-0.5);
\draw (4.2,0.3) node[anchor=center] {{\small $\otimes ~ \{ b_3, \{ b_2, \{ b_1, b_4 \} \} \} $}};
\end{scope}
\begin{scope}[shift={(0,0)}]
\draw (-2,0.3) node[anchor=center] {{$- \frac{1}{72}$}};
\node [n] (n) at (0,0) {};
\node [lab] (v1) at (-0.7,0.7) {\tiny $1$};
\node [lab] (v3) at (0,0.7) {\tiny $3$};
\node [lab] (v4) at (0.7,0.7) {\tiny $4$};
\node [lab] (v2) at (-1.2,1.2) {\tiny $2$};
\draw (n) edge (v1) edge (v3) edge (v4) edge (0,-0.5) (v1) edge (v2);
\draw (4.2,0.3) node[anchor=center] {{\small $\otimes ~ \{ b_3, \{ b_2, \{ b_1, b_4 \} \} \} $}};
\end{scope}
\begin{scope}[shift={(10,0)}]
\draw (-2,0.3) node[anchor=center] {{$- \frac{1}{72}$}};
\node [n] (n) at (0,0) {};
\node [lab] (v1) at (-0.4,0.6) {\tiny $1$};
\node [lab] (v2) at (0.4,0.6) {\tiny $2$};
\node [lab] (v3) at (0.8,1.2) {\tiny $3$};
\node [lab] (v4) at (1.2,1.8) {\tiny $4$};
\draw (n)  edge (v1) edge (v2) edge (0,-0.5) (v3) edge (v2) edge (v4);
\draw (4.2,0.3) node[anchor=center] {{\small $\otimes ~ \{ b_2, \{ b_1, \{ b_3, b_4 \} \} \} $}};
\end{scope}
\begin{scope}[shift={(20,0)}]
\draw (-2,0.3) node[anchor=center] {{$+ \frac{1}{72}$}};
\node [n] (n) at (0,0) {};
\node [lab] (v1) at (-0.7,0.7) {\tiny $1$};
\node [lab] (v2) at (0,0.7) {\tiny $2$};
\node [lab] (v4) at (0.7,0.7) {\tiny $4$};
\node [lab] (v3) at (0,1.4) {\tiny $3$};
\draw (n) edge (v1) edge (v2) edge (v4) edge (0,-0.5) (v2) edge (v3);
\draw (4.2,0.3) node[anchor=center] {{\small $\otimes ~ \{ b_3, \{ b_2, \{ b_1, b_4 \} \} \} $}};
\end{scope}
\begin{scope}[shift={(0,-3.5)}]
\draw (-2,0.3) node[anchor=center] {{$+\frac{5}{288}$}};
\node [n] (n) at (0,0) {};
\node [lab] (v1) at (-0.5,0.7) {\tiny $1$};
\node [lab] (v2) at (-1,1.4) {\tiny $2$};
\node [lab] (v3) at (0,1.4) {\tiny $3$};
\node [lab] (v4) at (0.5,0.7) {\tiny $4$};
\draw (n) edge (v1)  edge (v4) edge (0,-0.5) (v1) edge (v2) edge (v3);
\draw (4.2,0.3) node[anchor=center] {{\small $\otimes ~ \{ b_3, \{ b_2, \{ b_1, b_4 \} \} \} $}};
\end{scope}
\begin{scope}[shift={(10,-3.5)}]
\draw (-2,0.3) node[anchor=center] {{$-\frac{1}{720}$}};
\node [n] (n) at (0,0) {};
\node [lab] (v1) at (-0.5,0.5) {\tiny $1$};
\node [lab] (v2) at (-1,1) {\tiny $2$};
\node [lab] (v3) at (-1.5,1.5) {\tiny $3$};
\node [lab] (v4) at (0.5,0.5) {\tiny $4$};
\draw (n)  edge (v1) edge (0,-0.5) edge (v4) (v2) edge (v1) edge (v3);
\draw (4.2,0.3) node[anchor=center] {{\small $\otimes ~ \{ b_3, \{ b_2, \{ b_1, b_4 \} \} \} $}};
\end{scope}
\begin{scope}[shift={(20,-3.5)}]
\draw (-2,0.3) node[anchor=center] {{$-\frac{1}{36}$}};
\node [n] (n) at (0,0) {};
\node [lab] (v1) at (-0.7,0.7) {\tiny $1$};
\node [lab] (v2) at (0,0.7) {\tiny $2$};
\node [lab] (v3) at (0.7,0.7) {\tiny $3$};
\node [lab] (v4) at (1.3,1.3) {\tiny $4$};
\draw (n)  edge (v1) edge (v2) edge (v3) edge (0,-0.5) (v3) edge (v4);
\draw (4.2,0.3) node[anchor=center] {{\small $\otimes ~ \{ b_3, \{ b_2, \{ b_1, b_4 \} \} \} $}};
\end{scope}
\begin{scope}[shift={(0,-7)}]
\draw (-2,0.3) node[anchor=center] {{$-\frac{1}{72}$}};
\node [n] (n) at (0,0) {};
\node [lab] (v1) at (-0.5,0.7) {\tiny $1$};
\node [lab] (v3) at (0.5,0.7) {\tiny $3$};
\node [lab] (v2) at (-1,1.4) {\tiny $2$};
\node [lab] (v4) at (1,1.4) {\tiny $4$};
\draw (n)  edge (v1) edge (v3) edge (0,-0.5) (v1) edge (v2) (v3) edge (v4);
\draw (4.2,0.3) node[anchor=center] {{\small $\otimes ~ \{ b_3, \{ b_2, \{ b_1, b_4 \} \} \}. $}};
\end{scope}
\end{tikzpicture}
\caption{The last part of $\al_4$} \label{fig:lastPart}
\end{figure}

~~~~~
~~~~~
~~~~~

~\\

\noindent\textsc{Department of Mathematics,
Temple University, \\
Wachman Hall Rm. 638\\
1805 N. Broad St.,\\
Philadelphia PA, 19122 USA \\
\emph{E-mail addresses:} {\bf vald@temple.edu}, {\bf geoffrey.schneider@temple.edu}}

\end{document}